\documentclass[12pt,sans]{article}

\usepackage[margin=1in]{geometry}
\usepackage{microtype}
\usepackage{url}
\usepackage[colorlinks = true,
            linkcolor = blue,
            urlcolor  = blue,
            citecolor = blue,
            anchorcolor = blue]{hyperref}
\usepackage[utf8]{inputenc} 
\usepackage[T1]{fontenc}    
\usepackage{lmodern}
\usepackage{mathrsfs}
\usepackage[sans]{dsfont}   
\usepackage{booktabs}       
\usepackage{nicefrac}       
\usepackage{amsmath,amsthm,amssymb,amsfonts}   
\usepackage{latexsym}
\usepackage{optidef}
\usepackage{graphicx}
\usepackage{caption}
\usepackage{subcaption}
\usepackage{epstopdf}
\usepackage{cleveref}
\usepackage{color}
\usepackage{algorithm,algpseudocode}
\usepackage{tcolorbox}
\usepackage{xcolor} 
\newtcolorbox{tbox}[3][]{%
colframe=#2,colback=#2!10,coltitle=#2!20!black,title={#3},#1}
\usepackage{accents}
\usepackage{enumitem}
\usepackage{tikz-cd} 
\tikzset{smalltext/.style={"\textup{\small #1}" description}}
\usetikzlibrary{fit,backgrounds,arrows.meta}
\usepackage{silence}
\usepackage{autonum} 

\newcommand{\esp}{\mathbb{E}}
\newcommand{\espn}{\mathbf{E}}
\newcommand{\prob}{\mathbb{P}}
\newcommand{\probn}{\mathbf{P}}
\DeclareMathOperator{\var}{\mathbb{V}}

\newcommand{\unit}{\mathbf{1}}
	
\newcommand{\re}{\mathbb{R}}
\DeclareMathOperator{\tr}{tr}
\DeclareMathOperator{\erank}{r}

\DeclareMathOperator{\Count}{Count}

\def\bfGamma{\boldsymbol{\Gamma}}
\def\bfSigma{\boldsymbol{\Sigma}}
\def\bfTheta{\boldsymbol{\Theta}}

\def\bfI{\mathbf{I}}

\def\bfT{\mathbf{T}}
\def\bfU{\mathbf{U}}

\def\bx{\boldsymbol x}

\def\bu{\boldsymbol u}
\def\bv{\boldsymbol v}

\def\bz{\boldsymbol z}

\def\btheta{\boldsymbol\theta}

\def\calB{\mathcal B}

\def\calE{\mathcal E}
\def\calF{\mathcal F}

\def\calI{\mathcal I}

\def\calN{\mathcal N}

\def\calP{\mathcal P}
\def\calR{\mathcal R}

\def\calT{\mathcal T}

\def\sfT{\mathsf{T}}

\def\mbB{\mathbb{B}}

\def\mbS{\mathbb{S}}

\DeclareMathOperator*{\argmin}{argmin}
\DeclareMathOperator*{\argmax}{argmax}
\DeclareMathOperator{\Sym}{Sym}

\theoremstyle{plain}
\newtheorem{theorem}{Theorem}
\newtheorem{proposition}[theorem]{Proposition}
\newtheorem{lemma}[theorem]{Lemma}
\newtheorem{corollary}[theorem]{Corollary}
\newtheorem{assumption}[theorem]{Assumption}

\theoremstyle{definition}

\theoremstyle{remark}

\newtheorem{remark}{Remark}

\numberwithin{equation}{section}

\title{Robust dimension-free estimation of simple random tensors: optimal guarantees under heavy tails and adversarial contamination}
\date{}

\author{
Roberto I. Oliveira \and
Zoraida F. Rico \and
Philip Thompson
}

\newcommand{\Addresses}{{
  \bigskip
  \footnotesize

Roberto I.~Oliveira, \textsc{IMPA, Brazil}\par\nopagebreak
  \textit{E-mail address}: \texttt{rbimfo@impa.br} 

\and 

Zoraida F.~Rico, \textsc{Bocconi University, Italy}\par\nopagebreak
  \textit{E-mail address}: \texttt{fr.zoraida@gmail.com} 

\and

  P.~Thompson, \textsc{FGV EMAp, School of Applied Mathematics, Brazil}\par\nopagebreak
  \textit{E-mail address}: \texttt{philip.thompson@fgv.br}

  \medskip
}}

\begin{document}

\maketitle
\Addresses

\begin{abstract}
  We study robust estimation of simple random tensors of arbitrary order $q\in\mathbb{N}$ under finite-moment assumptions and adversarial contamination.  We propose the first robust estimator achieving near-optimal dimension-free statistical rates in this setting. The estimator attains the near-optimal corruption rate whenever $p\ge2q$ moments are finite and continues to provide nontrivial guarantees throughout the weak-moment regime $q\le p\le2q$. Being based on directional trimmed means and minimax aggregation, our estimator is adaptive to $p$ and upper bounds on hypercontractive constants without resorting to interval-intersection procedures. Our analysis extends the trimmed-mean framework underlying recent advances in robust mean and covariance estimation to arbitrary tensor order. In particular, we establish concentration inequalities for higher-order counting and truncated empirical multi-vector product processes. We believe these inequalities could be of independent interest beyond the present application, including algorithmic robust estimation.
\end{abstract}




\section{Introduction}

Tensors provide a natural framework for representing high-order interactions in modern data. They arise throughout statistics, machine learning, signal processing and optimization, where moment tensors, tensor decompositions and multilinear methods play central roles. We refer to the surveys of \cite{2021Lim, 2025Ballard:Kolda, 2021Bi:Tang:Yuan:Zhang:Qu, 2025Auddy:Xia:Yuan} for comprehensive overviews of tensor methods and applications.

The focus of this work is the estimation of simple (moment) tensors of arbitrary order associated with a random vector \cite{2020Vershynin,2024Zhivotovskiy}. Let $q\in\mathbb{N}$ and $\bx\in\re^d$ be a centered random vector. Define the multilinear form 
\begin{align}
  \bx^{\otimes q}(\bu_1,\ldots,\bu_q) := 
  \langle\bx,\bu_1\rangle\cdots\langle\bx,\bu_q\rangle, 
  \quad \forall(\bu_1,\ldots,\bu_q)\in(\re^d)^q. 
\end{align}
The \emph{q}th-order simple tensor is the multilinear form 
$
  \bfT := \esp[\bx^{\otimes q}],
$
assuming it is well-defined. A notable example is the covariance operator $\bfSigma:=\esp[\bx^{\otimes2}]$, which we identify with the usual covariance matrix. Given an i.i.d. sample $\{\bx_i\}_{i=1}^n$, a natural estimator of $\bfT$ is the \emph{empirical $q$th-order simple tensor} 
$$
  \widehat{\bfT} := \frac1n\sum_{i=1}^n \bx_i^{\otimes q}.
$$
Simple moment tensors encode higher-order dependence beyond covariance and play an important role, e.g., in latent-variable models, tensor decomposition and method-of-moments estimation. 

In the case of a subgaussian random vector, the estimation of the covariance matrix and higher order simple tensors (in operator norm) by means of the empirical tensor reduces to obtaining concentration inequalities for the \emph{$q$th-order empirical process}\footnote{In above, $\|\cdot\|_2$ denotes the Euclidean norm in $\re^d$.}
\begin{align}
  \epsilon_q(\bx_1,\ldots,\bx_n) := \sup_{\|\bu\|_2=1}\left|
    \frac{1}{n}\sum_{i=1}^{n}\langle \bx_i,\bu\rangle^{q}
    - \esp[\langle \bx,\bu\rangle^{q}]
  \right|.
\end{align}
(Quite often, this is given in terms of the absolute moment $|\langle \bx,\bu\rangle|^q$ instead). The development of sharp concentration inequalities for $\epsilon_q$ has a rich literature, spanning both classical asymptotic and recent nonasymptotic results. For covariance estimation we refer, e.g., to \cite{1999Rudelson,2005Klartag:Mendelson, 2007MendelsonPajorTomczak-Jaegermann, 2010Adamczak:Litvak:Pajor:Tomczak-Jaegermann, 2010Mendelson, 2012Mendelson:Paouris, 2012Vershynin, 2013Srivastava:Vershynin, 2014Mendelson:Paouris}. A central objective has been to understand the dependence of these inequalities on the ambient dimension. Remarkably, other papers such as \cite{2014Lounici}, \cite{2017Koltchinskii:Lounici} establish the first optimal \emph{dimension-free} rates in terms of the \emph{effective rank}
$$
\erank(\bfSigma):=\frac{\tr(\bfSigma)}{\Vert\bfSigma\Vert}.
$$

As one might suspect, going beyond the 2nd order creates difficulties, even in the subgaussian case. In fact:
\begin{quote}
  \emph{In sharp contrast with covariance estimation $(q=2)$, the empirical tensor is not the best estimator one can hope for when $q>2$, even for Gaussian distributions and clean iid samples.}
\end{quote}
We refer to a detailed discussion in that regard in the recent works \cite{2021Mendelson,2025Al-Ghattas:Chen:Sanz-Alonso:SharpConcentration, 2025Al-Ghattas:Chen:Sanz-Alonso,2026Bartl:Mendelson}. Classical works on empirical simple tensors and their associated concentration inequalities include \cite{2007Guedon:Rudelson, 2008Mendelson, 2010Adamczak:Litvak:Pajor:Tomczak-Jaegermann, 2011Vershynin, 2012Vershynin, 2013Srivastava:Vershynin}. For more recent results we refer to \cite{2021Even:Massoulie,2024Zhivotovskiy,2025Al-Ghattas:Chen:Sanz-Alonso:SharpConcentration, 2025Al-Ghattas:Chen:Sanz-Alonso,2025Abdalla:Vershynin,2025Shang}.

Motivated by recent developments in the robust statistics literature, another line of research studies estimation beyond the subgaussian setting and design estimators with minimal distributional assumptions. The goal is to obtain (near-)optimal estimators under heavy tails and adversarial contamination, where classical empirical process techniques fail. For robust covariance estimation we refer, e.g., to \cite{2018Minsker, 2018Tikhomirov, 2020Minsker:Wei, 2019Ostrovskii:Rudi} with recent advances in \cite{2018Catoni:Giulini, 2020Mendelson:Zhivotovskiy,2026Abdalla:Zhivotovskiy, 2024Oliveira:Rico}. 

To our knowledge, only two previous works address robust estimation of simple random tensors \cite{2021Mendelson,2026Bartl:Mendelson}. We compare our results with these and the aforementioned works in Section \ref{ss:related:work}.

In this work we are motivated by the following questions:
\begin{enumerate}
  \item What is the optimal robust statistical rate for  simple random tensors (under heavy tails and adversarial contamination)?
  \item Can the recent dimension-free methodology developed for robust covariance estimation be extended to simple tensors of arbitrary order?
  \item Recent work on robust mean and covariance estimation has shown that \emph{counting} and \emph{truncated} empirical processes play a central role in the analysis of trimmed-mean estimators \cite{2021Lugosi:Mendelson,2020Mendelson:Zhivotovskiy,2024Oliveira:Rico,2026Abdalla:Zhivotovskiy}. Can these processes and the underlying methodology be generalized to arbitrary tensor order?
\end{enumerate}
Our main theorem answers these questions in the affirmative by providing a universal robust estimator for simple tensors of arbitrary order together with a unified analysis of the associated counting and truncated empirical processes.

\subsection{Main result}

We begin by introducing the assumptions used throughout the paper.

\begin{assumption}[Finite covariance]\label{assump:finite:covariance:main}
    Let $\bx_1,\ldots,\bx_{n}$ be an i.i.d. $n$-sized sample of a centered random vector $\bx\in\re^d$ with distribution $\probn$ and population covariance matrix $\bfSigma\neq\mathbf{0}$.
\end{assumption}

\begin{assumption}[Adversarial contamination model]\label{assump:contamination:model}
    Suppose Assumption \ref{assump:finite:covariance:main} holds. Let $\tilde\bx_1,\ldots,\tilde\bx_{n}$ be arbitrary random elements of $\re^d$ and $\epsilon\in[0,1/2)$ such that 
    $$
        \#\{i\in[n] : \tilde\bx_i\neq \bx_i\} \le \epsilon n.
    $$
\end{assumption}

The number $\epsilon$ is referred to as the \emph{contamination rate}. Assumption \ref{assump:contamination:model} is quite general, as it allows the adversary to choose the contaminated observations after observing the uncontaminated sample \cite{2019Diakonikolas:Kamath:Kane:Li:Moitra:Stewart,2021Lugosi:Mendelson}. This is strictly stronger than Huber's contamination model, where the contaminating distribution is fixed independently of the sample \cite{2018Chen:Gao:Ren}.

Throughout the paper, moment growth is measured through the $L^s$--$L^2$ hypercontractive constants. For every $s\ge2$, define
\[
\kappa_s
:=
\sup_{\Vert\bfSigma^{1/2}\bv\Vert_2=1}
\left(
\esp\big[
|\langle\bx,\bv\rangle|^s
\big]
\right)^{1/s}.
\]
The quantity $\kappa_s$ is commonly referred to as the \emph{$L^s$--$L^2$ hypercontractive constant}.

We also make the following assumption. 

\begin{assumption}[Rank estimation]\label{assump:trace:estimator}
  Suppose Assumption \ref{assump:contamination:model} holds. There exist a constant $C_{\probn}>0$, depending only on $\probn$, and a universal constant $c_0\in(0,1)$ such that, for any $\delta\in(0,c_0)$, $n\ge C_{\probn}(\erank(\bfSigma)\vee\log(1/\delta))$ and $\epsilon\le1/C_{\probn}$, there is an estimator $\widehat{\erank}:(\re^d)^n\rightarrow (0,\infty)$, depending only on $(n,\epsilon,\delta)$, such that with probability at least $1-\delta$, 
  $
    \frac{\erank(\bfSigma)}{3}\le\widehat{\erank}(\tilde\bx_1,\ldots,\tilde\bx_n) \le 3\erank(\bfSigma). 
  $
\end{assumption}

Assumption \ref{assump:trace:estimator} isolates the only point in our analysis where additional information on the underlying distribution is required. Existing robust rank estimators satisfy this assumption with $C_{\probn}\asymp\kappa_4^4$ under fourth-moment assumptions \cite{2020Mendelson:Zhivotovskiy,2026Abdalla:Zhivotovskiy,2024Oliveira:Rico}. Under more regular distributions this constant can be significantly smaller.\footnote{This includes covariance matrices with bounded effective rank or isotropic vectors with only $(2+\alpha)$ finite moments, for which $C_{\probn}\asymp1$ even though $\kappa_4=\infty$.}

We are now ready to state our main theorem. Let $\Sym_q^d(\re^d)$ denote the space of symmetric tensors of order $q$ over $\re^d$, endowed with its operator norm $\|\cdot\|$. Throughout, for $a,b\ge0$ and $s>0$, we write
$
a\lesssim_s b
$
to denote that $a\le C_sb$ for an absolute constant $C_s$ depending only on $s$, and
$
a\asymp_s b
$
whenever both $a\lesssim_s b$ and $b\lesssim_s a$ hold.

\begin{theorem}[Main result]\label{thm:main}
  Let $q\in\mathbb{N}$, $q\ge2$. Suppose Assumptions \ref{assump:finite:covariance:main}, \ref{assump:contamination:model}  and \ref{assump:trace:estimator} hold and there is $p \ge2\vee q$ such that $\kappa_p<\infty$. There is constant $C\ge C_{\probn}$ and absolute constant $c\in(0,c_0)$ such that the following holds. Fix any $\delta\in(0,c)$, $n\ge C(\erank(\bfSigma)\vee \log(4/\delta))$ and $\epsilon\le1/C$. Then there is an estimator $\widehat\bfT_*:(\re^d)^n\rightarrow \Sym_{d}^q(\re^d)$, depending only on $(n, \epsilon,\delta)$ and the estimator $\widehat{\erank}:(\re^d)^n\rightarrow (0,\infty)$, such that with probability at least $1-\delta$, if $p\ge2q$,
  \begin{align}
    \|\widehat\bfT_*(\tilde\bx_{1},\ldots,\tilde\bx_n)-\esp[\bx^{\otimes q}]\|
    \lesssim_q \kappa_{2q}^{q}\|\bfSigma\|^{\frac{q}{2}}
    \sqrt{\frac{\erank(\bfSigma)\vee\log(4/\delta)}{n}}
    + c_{p,q}(\epsilon)\cdot\kappa_{p}^{q}\|\bfSigma\|^{\frac{q}{2}}
    \epsilon^{1 - \frac{q}{p}},
  \end{align}
  with 
  $
    c_{p,q}(\epsilon) := 
    \left(\min\{ p, \log\left(\frac{q}{\epsilon} \right) \}\right)^{\frac{q}{2}}
    \lesssim_q \log^{\frac{q}{2}}\left(\frac{1}{\epsilon} \right).
  $
  On the same event, if $2\vee q \le p < 2q$,
  \[
    \|\widehat\bfT_*(\tilde\bx_{1},\ldots,\tilde\bx_n)-\esp[\bx^{\otimes q}]\|
    \lesssim_q \kappa_{p}^{q} \|\bfSigma\|^{\frac{q}{2}}\left(
      \frac{\erank(\bfSigma)\vee\log(4/\delta)}{n}
    \right)^{1-\frac{q}{p}} 
    + \kappa_{p}^{q}\|\bfSigma\|^{\frac{q}{2}}\epsilon^{1 - \frac{q}{p}}.
\]
\end{theorem}

Theorem \ref{thm:main} is complemented by the following minimax lower bound under an $L^p$--$L^2$ hypercontractivity assumption. Its proof, given in Section \ref{ss:supplement:prop:optimal:hypercontractive} of the supplement, relies on standard lower-bound techniques \cite{2018Chen:Gao:Ren,2018Minsker:uniform,2025Minsker}; see also \cite{2021Lugosi:Mendelson,2026Abdalla:Zhivotovskiy} for related constructions.

\begin{proposition}[Lower bound]\label{prop:optimal:hypercontractive}
Let $p\ge q\ge1$, $\kappa>1$. Define 
\[
\calF_{p,\kappa}
:=
\left\{
P\in\calP(\re):
\esp_P|X|^2\ge1,
\{\esp_P|X|^p\}^{1/p}
\le
\kappa\{\esp_P|X|^2\}^{1/2}
\right\},
\]
where $\calP(\re)$ denotes the class of distributions on $\re$. Then there exist constant $c>0$ such that, for every $n\in\mathbb{N}$ and $\epsilon\in[0,1/2)$ such that 
$
\frac{\epsilon}{1-\epsilon}
\le 2^{-\frac{p}{q}}(\kappa^p-1),
$
\[
\inf_{\widehat T_q}
\sup_{P\in\calF_{p,\kappa}}
\sup_{Q\in\calP(\re)}
\prob_{\left((1-\epsilon)P+\epsilon Q\right)^{\otimes n}}
\left(
\left|
\widehat T_q-\esp_P|X|^q
\right|
\ge
\frac{1}{2}(\kappa^p-1)^{q/p}
\epsilon^{1-\frac{q}{p}}
\right)
\ge c,
\]
where the inf is over all estimators. 
\end{proposition}

Combining Theorem \ref{thm:main}, Proposition \ref{prop:optimal:hypercontractive}, and existing lower bounds yields near-minimax optimality (up to constants) whenever $p\ge2q$:
\begin{enumerate}
  \item By \cite{2024Zhivotovskiy,2025Al-Ghattas:Chen:Sanz-Alonso,2025Al-Ghattas:Chen:Sanz-Alonso:SharpConcentration}, the term $\|\bfSigma\|^{q/2}\sqrt{\frac{\erank(\bfSigma)\vee\log(1/\delta)}{n}}$ cannot be improved, up to the constant factor $\kappa_{2q}^q$. 
  \item By Proposition \ref{prop:optimal:hypercontractive}, the optimal dependence on the contamination level $\epsilon$ is of order
$
\epsilon\log^{q/2}\!\left(\nicefrac1\epsilon\right).
$
Our estimator therefore loses only by the dimension-free logarithmic factor
$
\log^{q/2}\!\left(\nicefrac1\epsilon\right).
$
In particular, for subgaussian distributions the contamination term in Theorem \ref{thm:main} is of order
$
\epsilon\log^{q}\!\left(\nicefrac1\epsilon\right),
$
whereas the minimax rate is
$
\epsilon\log^{q/2}\!\left(\nicefrac1\epsilon\right).
$
For distributions satisfying only a finite $p$-moment assumption, our estimator achieves the optimal rate up to absolute constants.
\end{enumerate}

In particular, Theorem \ref{thm:main} recovers robust covariance estimation ($q=2$) as a special case \cite{2026Abdalla:Zhivotovskiy,2024Oliveira:Rico}.

\paragraph{Our contributions.}
We develop the first dimension-free robust estimator of simple random tensors achieving near-optimal statistical rates under finite-moment assumptions and adversarial contamination. To the best of our knowledge, this is also the first robust estimator providing statistical guarantees throughout the weak-moment regime $2\vee q\le p\le2q$. The analysis is unified across all tensor orders and recovers the optimal results for covariance  estimation ($q=2$) as a special case \cite{2026Abdalla:Zhivotovskiy,2024Oliveira:Rico}.

From a probabilistic perspective, for any $q\in\mathbb{N}$, we establish concentration inequalities for higher-order counting and truncated empirical multi-vector product processes. These results extend the counting and truncated empirical-process framework recently developed for robust mean and covariance estimation to arbitrary tensor order. We expect these inequalities to have further applications in algorithmic robust estimation \cite{2019Diakonikolas:Kamath:Kane:Li:Moitra:Stewart,2020Diakonikolas:Kane:Pensia}. In Section~\ref{s:counting:truncated:processes}, we define the aforementioned higher-order counting and truncated empirical processes. Their concentration inequalities are stated and proved in Section~\ref{s:counting} for the counting process and in Section~\ref{s:truncation} for the truncated process.

From a methodological perspective, our estimator is fully adaptive: it requires neither knowledge of the moment parameter $p$ nor upper bounds on the hypercontractive constants, and avoids Lepski-type procedures and interval-intersection schemes.

\paragraph{The estimator.} Fix a \emph{trimming parameter} $k\in[n]$ and direction $\bu\in\re^d\setminus\{\mathbf{0}\}$. The \emph{trimmed mean} evaluation of the empirical $q$-order tensor at $\bu$ is
\begin{align}
\sfT_k(\tilde\bx_{1},\ldots,\tilde\bx_{n}\mid\bu)
:=
\frac{1}{n-2k}
\sum_{i=k+1}^{n-k}
\langle\tilde\bx,\bu\rangle_{(i)}^{\,q},
\qquad
1\le k<\frac n2,
\label{label:trimmed:mean:tensor:u}
\end{align}
where 
$
  \langle\tilde\bx,\bu\rangle_{(1)}^q
  \le \cdots
  \le \langle\tilde\bx,\bu\rangle_{(n)}^q
$
denote the ordered values of $\{\langle\tilde\bx_i,\bu\rangle^q\}_{i=1}^n$. 

Define the minimax directional aggregation estimator by
\begin{align}
    \widehat\bfT_k(\tilde\bx_{1},\ldots,\tilde\bx_{n}) \in \argmin_{\bfT\in\Sym^d_{q}(\re^d)}
    \sup_{\bu\in\mbS_2}\left|
      \langle\bfT,\bu^{\otimes q}\rangle - \sfT_k(\tilde\bx_{1},\ldots,\tilde\bx_{n}|\bu)
    \right|,
    \label{label:minimax:trimmed:estimator}
\end{align}
where $\mbS_2$ denotes the Euclidean sphere on $\re^d$. Our estimator is given by 
\begin{align}
  \widehat\bfT_*:=\widehat\bfT_{\widehat k} 
  \quad \mbox{ with }\quad 
  \widehat k := \left\lceil \max\Big\{3A_1\widehat\erank,~A_2\log\left(\nicefrac{4}{\delta}\right),~A_3(\epsilon n)\Big\}\right\rceil
  \wedge\left\lfloor\frac{n-1}{2}\right\rfloor.
  \label{label:minimax:trimmed:estimator:hat}
\end{align}
In above, $A_1,A_2,A_3\ge1$ are known constants depending only on the constant $C\ge C_{\probn}$ stated in Theorem \ref{thm:main}. The existence of such $C$ ensures that $1\le\hat k<n/2$ as necessary. See Section \ref{ss:thm:main:proof} for the proof of Theorem \ref{thm:main} and the details.

\paragraph{Overview of the proof.} The proof of Theorem \ref{thm:main} is inspired, with proper adaptations, by the method developed in \cite{2024Oliveira:Rico} tailored to robust covariance estimation. An important initial step is to relate the trimmed mean with the \emph{truncated mean}. 
Given a \emph{truncation threshold} $Q>0$, define the truncation function
\[
  \psi_Q(t) := \max\{\min\{t,Q\},-Q\}. 
\]
The truncated mean analog of \eqref{label:trimmed:mean:tensor:u} is 
\begin{align}
  \sfT_Q(\tilde\bx_1,\ldots,\tilde\bx_{n}|\bu) &:= \frac{1}{n}\sum_{i=1}^n\psi_Q(\langle\tilde\bx_{i},\bu\rangle^q).\label{label:truncated:mean:tensor:u}
\end{align}
The fact that trimmed and truncated means can be uniformly approximated over any direction follows from the so called \emph{counting condition}:
\begin{align}
  \#\{i\in[n]:|\langle\tilde\bx_i,\bu\rangle|\ge Q^{1/q}\} \le k,\quad\forall\bu\in\mbS_2. 
\end{align}
(See Lemma \ref{lemma:trim:trunc} in Section \ref{s:robust:estimator}).

The proof relies on two concentration results. First, we establish the counting condition with high probability. Second, we derive a uniform concentration inequality for the truncated process. Both follow from a common strategy: Gaussian smoothing, a PAC-Bayesian Bernstein inequality for the smoothed process and ad hoc Gaussian residual estimates to bound the expectation of the process. Concentration follows from Bousquet's inequality.

\begin{remark}[Computational considerations]
The estimator $\widehat\bfT_k$ is not computationally efficient. As in \cite{2024Oliveira:Rico}, one may replace $\mbS_2$ in \eqref{label:minimax:trimmed:estimator} by a $\frac{1}{2q}$-net $\calN$, defining the estimator $\widehat\bfT_{\calN,k}$. A standard net argument yields
\[
\|\widehat\bfT_{\calN,k}-\esp[\bx^{\otimes q}]\|
\lesssim
2\sup_{\bu\in\calN}
\left|
\langle\bfT,\bu^{\otimes q}\rangle
-
\sfT_k(\tilde\bx_{1},\ldots,\tilde\bx_n|\bu)
\right|,
\]
which is sufficient for our analysis (see Section~\ref{s:robust:estimator}). The optimization problem defining $\widehat\bfT_{\calN,k}$ has a convex objective. Nevertheless, this approximation remains computationally intractable in general, since $|\calN|$ is exponential in $d$. In addition, optimization over symmetric tensors is already difficult for $q>2$ without additional structural assumptions.
\end{remark}

\subsection{Comparison with prior robust tensor estimators}\label{ss:related:work}

To our knowledge, \cite{2021Mendelson,2026Bartl:Mendelson} are the only previous works directly addressing robust estimation of simple tensors.

Under an $L^{p}\!-\!L^{q}$ hypercontractivity assumption $(p\ge2q)$ and i.i.d.\ sampling, Mendelson \cite{2021Mendelson} showed that one-dimensional moments of arbitrary order $q\ge1$ can be robustly estimated uniformly over all directions, with error of order
$
\sqrt{\left(\nicefrac{d}{n}\right)\log\!\left(\nicefrac{n}{d}\right)},
$
using trimmed means along one-dimensional marginals. The proof extends naturally to more general one-dimensional functions than $t\mapsto|t|^q$.
His work thereby extends the empirical tensor approximation approach of, e.g., \cite{2007Guedon:Rudelson, 2008Mendelson, 2010Adamczak:Litvak:Pajor:Tomczak-Jaegermann, 2011Vershynin, 2012Vershynin, 2013Srivastava:Vershynin} to heavy-tailed distributions. Compared with our results, the statistical rate incurs an additional logarithmic factor, is not dimension-free, and does not address adversarial contamination.

Bartl and Mendelson \cite{2026Bartl:Mendelson} subsequently proposed a robust estimator for simple tensors of arbitrary order under the assumption $\kappa_4<\infty$. Under i.i.d.\ sampling, their estimator attains the optimal dimension-free statistical rate and applies to a broader class of one-dimensional functionals, including non-integer powers $q\ge1$. Their approach, however, assumes knowledge of the $L^2$ norm induced by the distribution (Assumption~1.3 in \cite{2026Bartl:Mendelson}), which, for simple tensors, amounts to knowing the pseudo-norm $\|\bfSigma^{1/2}(\cdot)\|_2$.

Under adversarial contamination, however, their optimal corruption rate is established only under fourth-moment assumptions. Their construction combines generic chaining through an admissible sequence of nets with robust one-dimensional mean estimation, and whether this framework extends to the optimal corruption rates for $p>4$ remains open. In contrast, our estimator attains the optimal corruption rate for every $p\ge2q$, continues to provide convergent guarantees throughout the weak-moment regime $q\le p\le2q$, and relies only on directional trimmed means and minimax aggregation.

\subsection{Additional related work}

This section reviews related work not directly concerned with robust simple tensor estimation.

\paragraph{Robust covariance estimation.} Various robust covariance estimators with heavier tails have been proposed, including \cite{2018Minsker, 2018Tikhomirov, 2020Minsker:Wei, 2019Ostrovskii:Rudi}. Prior work also consider the Gaussian model with contamination \cite{2018Chen:Gao:Ren}. For recent advances we refer to, e.g., \cite{2018Catoni:Giulini, 2020Mendelson:Zhivotovskiy,2026Abdalla:Zhivotovskiy, 2024Oliveira:Rico}. Under bounded fourth moments, \cite{2018Giulini, 2018Catoni:Giulini, 2020Mendelson:Zhivotovskiy} obtain near-optimal rates up to logarithmic factors in $\erank(\bfSigma)$. The logarithmic factor is removed in \cite{2018Catoni:Giulini} under additional distributional knowledge and independently in \cite{2026Abdalla:Zhivotovskiy,2024Oliveira:Rico} by different techniques. Our trimmed-mean estimator is inspired by recent developments in robust mean \cite{2021Lugosi:Mendelson} and covariance estimation \cite{2026Abdalla:Zhivotovskiy,2024Oliveira:Rico}. More specifically, it extends, with the necessary adaptations, the robust covariance estimator of \cite{2024Oliveira:Rico} to arbitrary tensor order.

Probabilistically, our analysis extends the counting and truncated quadratic processes of \cite{2024Oliveira:Rico} to higher-order counting and truncated empirical multi-vector product processes. The proof relies on Gaussian smoothing with a product Gaussian measure and exploits coordinate independence. Compared with \cite{2024Oliveira:Rico}, our proof admits simplifications. We apply the PAC-Bayesian method only to bound the expectation of the smoothed processes, while concentration follows from Bousquet's inequality. Moreover, the Gaussian smoothing residual is controlled through a sharp, moment-independent bound based on the Lambert \(W\)-function, yielding a substantially smaller truncation threshold and consequently weaker moment requirements in Assumption~\ref{assump:trace:estimator}.\footnote{See Lemma \ref{lemma:exp:Q}. In \cite{2024Oliveira:Rico}, the residual is bounded using moment-dependent estimates based on the Gamma function.}

\paragraph{Subgaussian tensor estimation.}
Alternative proofs and sharpness discussions for subgaussian covariance estimation can be found in \cite{2017Liaw:Mehrabian:Plan:Vershynin,2017vanHandel,2022Jeong:Li:Plan:Yilmaz}. Notably, Bednorz and Dirksen independently extended Talagrand's generic chaining method to subgaussian quadratic processes in \cite{2014Bednorz,2015Dirksen}. For completeness, we also refer to \cite{2016Mendelson,2017Mendelson} for an earlier work on the extension of the generic chaining method for tails heavier than the subgaussian.  

Concentration inequalities for the empirical tensor process $\epsilon_q$ have been extensively studied in the subgaussian setting \cite{2007Guedon:Rudelson,2008Mendelson,2010Adamczak:Litvak:Pajor:Tomczak-Jaegermann,2011Vershynin}; see also \cite{2021Even:Massoulie}. Related concentration inequalities for more general random tensors and multilinear forms are developed in \cite{2020Vershynin,2021Gotze:Holger:Sambale:Sinulis,2023Sambale,2026Chen:Sanz-Alonso:SharpConcentrationAsymmetry}.

Recent work has substantially sharpened the dimension-free theory. Zhivotovskiy \cite{2024Zhivotovskiy} obtained the first effective-rank bounds for empirical tensors, achieving the rate
$
\|\bfSigma\|^{q/2}(\nicefrac{\erank(\bfSigma)}{n})^{1/2}
$
in the regime $n\gtrsim\erank^{\,q-1}(\bfSigma)$, up to a likely logarithmic factor in expectation and exactly in high probability. This was sharpened to the optimal rate for all sample sizes by \cite{2025Al-Ghattas:Chen:Sanz-Alonso:SharpConcentration,2025Al-Ghattas:Chen:Sanz-Alonso} using generic chaining, later simplified in \cite{2025Abdalla:Vershynin} and extended to $1<q<2$ in \cite{2025Shang}. Moreover, \cite{2025Al-Ghattas:Chen:Sanz-Alonso} shows that, for Gaussian distributions, an Isserlis plug-in estimator strictly outperforms the empirical tensor and is statistically optimal.

\subsection{Organization} Section \ref{s:preliminaries} introduces notation and preliminary results. In Section~\ref{s:counting:truncated:processes}, we define the aforementioned higher-order counting and truncated empirical processes. Their concentration inequalities are stated and proved in Section~\ref{s:counting} for the counting process and in Section~\ref{s:truncation} for the truncated process. Section \ref{s:robust:estimator} proves the main theorem. Technical lemmas, omitted proofs of some auxiliary results, and a proof dependency diagram are deferred to the Supplementary Material (Section~\ref{s:supplement}); see in particular Section~\ref{ss:proof:diagram} for the proof diagram.

\section{Preliminaries}\label{s:preliminaries}

\paragraph{Additional notation.} Throughout the paper, $C,c>0$ denote absolute constants that may change within the text. Given $Q>0$, we define the function:
$$
\calI_Q(t):=\unit_{|t|>Q}.
$$
For any $s\ge1$,
$
  \nu_s := \sup_{\bv\in\mbS_2}(\esp[|\langle\bx,\bv\rangle|^s])^{\frac{1}{s}}.
$
Note that $\nu_s \le \|\bfSigma\|^{1/2}\kappa_s$.

We use the notation $\bz_{1:n}:=\{\bz_1,\ldots,\bz_n\}$ for a finite sequence. $a\vee b:=\max\{a,b\}$ and $a\wedge b:=\min\{a,b\}$. $\mbS_2$ and $\mbB_2$ denote, respectively, the Euclidean unit sphere and ball in $\re^d$. $\probn_n$ and $\espn_n$ denote, respectively, the empirical distribution and expectation for $\bx_1,\ldots,\bx_n$.

\paragraph{PAC-Bayesian Bernstein inequality.}
We recall the inequality established in \cite{2024Oliveira:Rico}.

When $m,q\in\mathbb N$ are fixed, we reserve the notation
$
\bfU:=\{\bu_{j,\ell}\}_{j\in[m],\,\ell\in[q]},
$
and 
$
\bfTheta:=\{\btheta_{j,\ell}\}_{j\in[m],\,\ell\in[q]},
$
for collections of vectors in $\re^d$. We equip the product space
$(\re^d)^{m\times q}$ with the Euclidean product norm
$
\|\bfU\|_2
:=
(
\sum_{j=1}^m
\sum_{\ell=1}^q
\|\bu_{j,\ell}\|_2^2
)^{1/2},
$
and denote its unit ball by $\mbB_2^{m\times q}$.

A probability space $(\Omega,\calF,\prob)$ is implicit in our discussion. Consider a family of functions
\[
(\bfTheta,\omega)
\in
(\re^d)^{m\times q}\times\Omega
\longmapsto
X_i(\bfTheta,\omega)\in\re, 
\]
that are $\calB\left((\re^d)^{m\times q}\right)\otimes \calF/\calB(\re)$-measurable. We write $X_i(\bfTheta)$ for the measurable function
$\omega\mapsto X_i(\bfTheta,\omega)$.
For every fixed $\bfTheta\in(\re^d)^{m\times q}$, we assume
the random variables
$\{X_i(\bfTheta)\}_{i\in[n]}$
are i.i.d. and integrable.
We suppress the dependence on $\omega$ throughout.

Let $\Gamma_{\bv,\gamma}$
denote the Gaussian measure on $\re^d$
with mean $\bv$ and covariance
$\gamma^2\bfI_d$ for some $\gamma>0$. 
For $\bfU\in(\re^d)^{m\times q}$, define the product Gaussian measure
$
\Gamma_{\bfU,\gamma}
:=
\bigotimes_{j=1}^m
\bigotimes_{\ell=1}^q
\Gamma_{\bu_{j,\ell},\gamma}.
$
In what follows, we assume the integrals: 
\[
\Gamma_{\bfU,\gamma}X_i(\bfTheta,\omega)
:=
\int_{(\re^d)^{m\times q}}
X_i(\bfTheta,\omega)\,
\Gamma_{\bfU,\gamma}(d\bfTheta),
\]
are well-defined for all $\bfU\in(\re^d)^{m\times q}$ and depend continuously on $\bfU$. Again, we often have the dependence on $\omega$ implicit in our notation; observe, however, that under our assumptions the maps $(\bfU,\omega)\mapsto\bfGamma_{\bfU,\gamma}X_i(\bfTheta,\omega)$ are also $\calB\left((\re^d)^{m\times q}\right)\otimes \calF/\calB(\re)$-measurable.

The next proposition follows by integrating the tail bound in Proposition 2.2 of \cite{2024Oliveira:Rico}, applied to the product space $(\re^d)^{m\times q}$.
\begin{proposition}[Proposition 2.2 of \cite{2024Oliveira:Rico}]\label{prop:bernstein:smoothed}
Suppose there exists $\bar A>0$ such that
$
X_i(\bfTheta)-\esp[X_i(\bfTheta)]
\le \bar A
$
almost surely for every
$i\in[n]$
and
$\bfTheta\in(\re^d)^{m\times q}$.
Suppose also that
$
\bar\mu_\gamma
:=
\sup_{\bfU\in\mbB_2^{m\times q}}
\Gamma_{\bfU,\gamma}
\esp[X_1(\bfTheta)]
$
and
$
\bar\sigma_\gamma^2
:=
\sup_{\bfU\in\mbB_2^{m\times q}}
\Gamma_{\bfU,\gamma}
\var\!\left(X_1(\bfTheta)\right)
$
are finite. Then
\[
\esp\left[
\sup_{\bfU\in\mbB_2^{m\times q}}
\frac1n
\sum_{i=1}^n
\Gamma_{\bfU,\gamma}
X_i(\bfTheta)
\right]
\le
\bar\mu_\gamma
+
\bar\sigma_\gamma
\sqrt{\frac{\gamma^{-2}+2}{n}}
+
\bar A\,
\frac{\gamma^{-2}+2}{6n}.
\]

In addition, suppose that 
$
  X_i(\bfTheta)-\esp[X_i(\bfTheta)] \ge -\bar A
$
almost surely for every
$i\in[n]$
and
$\bfTheta\in(\re^d)^{m\times q}$. Then
\[
\esp\left[
\sup_{\bfU\in\mbB_2^{m\times q}}\left|
\frac1n
\sum_{i=1}^n
\Gamma_{\bfU,\gamma}
X_i(\bfTheta)
\right|
\right]
\le
\bar\mu_\gamma
+
\bar\sigma_\gamma
\sqrt{\frac{\gamma^{-2}+2(1+\log2)}{n}}
+
\bar A\,
\frac{\gamma^{-2}+2(1+\log2)}{6n}.
\]
\end{proposition}

Finally, we recall Bousquet's version of Talagrand's inequality for empirical processes \cite{2002Bousquet}; see also Theorem 12.5 and Corollary 12.2 in \cite{2013Boucheron:Lugosi:Massart}. We state only its Bernstein-type simplification. Its original result is stated for countable classes, but it extends directly to separable classes. See, e.g., page 315 in \cite{2013Boucheron:Lugosi:Massart} for the definition of a separable class and a related discussion. The function classes used in this work are separable.

\begin{theorem}[Bousquet's inequality]
\label{thm:bousquet}
Let $X_1,\ldots,X_n$ be i.i.d.\ random variables taking values in a measurable space $\mathcal X$, and let $\mathcal F$ be a separable class of measurable functions $f:\mathcal X\to\mathbb R$ such that, for every $f\in\mathcal F$,
$
  \esp[f(X_1)]=0
$
and
$
  f(X_1)\le1
$
almost surely. Assume furthermore that
$
    \bar\mu
    :=
    \esp\left[
        \sup_{f\in\mathcal F}
        \frac1n\sum_{i=1}^n f(X_i)
    \right]
    <\infty
$
and
$
    \bar\sigma^2
    :=
    \sup_{f\in\mathcal F}
    \frac1n\sum_{i=1}^n\esp[f(X_i)^2]
    <\infty.
$
Then, for every $t\ge0$, with probability at least $1-e^{-t}$,
$$
    \sup_{f\in\mathcal F}
    \frac1n\sum_{i=1}^n f(X_i)
    \le
    \bar\mu
    +
    \sqrt{\frac{2(2\bar\mu+\bar\sigma^2)t}{n}}
    +
    \frac{t}{3n}.
$$
\end{theorem}

\section{Higher-order counting and truncated processes}\label{s:counting:truncated:processes}

Although motivated by simple random tensor estimation, our methods extend almost verbatim to a slightly broader class of empirical multi-vector processes. Specifically, for any $m,q\in\mathbb{N}$ and $t>0$, we establish concentration inequalities for the \emph{counting process} 
\begin{align}
    Z_t := \sup_{\bfU\in\mbB_2^{m\times q}}
      \frac{1}{n}\sum_{i=1}^n
    \calI_t\left(
      \sum_{j=1}^{m}
      \prod_{\ell=1}^{q}
      \langle \bx_i,\bu_{j,\ell}\rangle
    \right),\label{equation:counting:process}
\end{align} 
and the \emph{truncated process} 
\begin{align}
    \epsilon_t := \sup_{\bfU\in\mbB_2^{m\times q}}
    \left|
      \frac1n\sum_{i=1}^n
      \psi_t\!\left(
      \sum_{j=1}^{m}
      \prod_{\ell=1}^{q}
      \langle\bx_i,\bu_{j,\ell}\rangle
      \right)
      -
      \esp\!\left[
      \psi_t\!\left(
      \sum_{j=1}^{m}
      \prod_{\ell=1}^{q}
      \langle\bx,\bu_{j,\ell}\rangle
      \right)
      \right]
    \right|.
    \label{equation:truncated:process}
\end{align}
We also define the \emph{truncation bias}
\begin{align}
  \calT_t := \sup_{\bfU\in\mbB_2^{m\times q}}\left|
    \esp\left[
      \left(
      \sum_{j=1}^{m}
      \prod_{\ell=1}^{q}
      \langle\bx,\bu_{j,\ell}\rangle
      \right)
      -\psi_{t}\!\left(
      \sum_{j=1}^{m}
      \prod_{\ell=1}^{q}
      \langle\bx,\bu_{j,\ell}\rangle
      \right)\right]
    \right|.
  \label{equation:truncation:bias}
\end{align}

The concentration inequalities of the counting and truncated processes are presented, respectively, in Sections \ref{s:counting} and \ref{s:truncation} in the following. 

\begin{remark}
  To prove Theorem \ref{thm:main}, we will only need $m=1$. For this theorem there is also an alternative proof with identical directions $\bu_{j,\ell}\equiv\bu$. We work with $m>1$ and multi-directions $\bu_{j,\ell}$ because we believe these inequalities are useful in other contexts. The proof with multi-directions is also simpler since it exploits coordinate independence in the Gaussian smoothing computations. 
\end{remark}

\begin{remark}\label{rem:logical:constants}
Throughout Sections~\ref{s:counting} and \ref{s:truncation}, we assume that $q\in\mathbb{N}$, 
$p\ge2$ with $\kappa_p<\infty$, $r\in[q,2q]$ with $\kappa_r<\infty$,
$k\in[n]$, and $c_1\in(0,1/e)$ satisfies
$
\sqrt{e}\,\frac{c_1k}{n}\le1,
$
and that $A\ge1$ and $\gamma,c_2,c_3>0$ are fixed.
\end{remark}

\section{Counting \& related arguments}\label{s:counting}

The main goal of this section is to establish the concentration inequality for the counting process, presented in Section~\ref{ss:counting:process}. To this end, we also establish two auxiliary results: a counting lemma (Lemma~\ref{lemma:count:linear:v2}) and a lemma controlling a smoothing residual (Lemma~\ref{lemma:exp:Q}), both presented in Section~\ref{ss:counting:directions}. These two lemmas are also used to establish concentration of the truncated process in Section~\ref{s:truncation}.

\subsection{Counting lemma in expectation \& smoothing residual}\label{ss:counting:directions}

Given any $\bz\in\re^d$, we define the quantity 
$$
  \|\bfU\|_{\infty}(\bz) := \max_{j\in[m],\,\ell\in[q]}
    |\langle\bz,\bu_{j,\ell}\rangle|. 
$$
Fix $B>0$. In this section, we bound 
$$
  \esp\left[\sup_{\bfU\in\mbB_2^{m\times q}}
  \frac1n\sum_{i=1}^n
  \unit_{\{
    \|\bfU\|_{\infty}(\bx_i)
    >B
  \}}\right].
$$
This is the content of Lemma \ref{lemma:count:linear:v2} in the following. 

Let 
$
  \sigma_p := (\esp|N|^p)^{1/p}
$
denote the $L_p$-norm of a standard normal variable $N\sim\calN(0,1)$. It is well known that $\sigma_p\asymp\sqrt p$.
\begin{lemma}[Counting lemma in expectation; proof in \S \ref{proof:lemma:count:linear:v2}]
  \label{lemma:count:linear:v2}
  Suppose
  \begin{align}
    \frac{B}{2} &\ge 
    \left(
      \nu_p^p\frac{n}{c_1k}
    \right)^{\frac{1}{p}}
    \bigvee 
    \left(\gamma \tr^{1/2}(\bfSigma)
      \left(
        \kappa_p^p\frac{n}{c_1k}
      \right)^{\frac1p}
      \sqrt{\min\left\{\sigma_p^2,
      4 \log\left( \frac{n}{c_1k} \right)
    \right\}}\right). \label{lemma:count:linear:v2:condition}
  \end{align}
  Then
  $$
      \esp\left[
        \sup_{\bfU\in\mbB_2^{m\times q}}
        \frac{1}{n}\sum_{i=1}^n\unit_{\{
          \|\bfU\|_{\infty}(\bx_i)>B \}}
      \right]
    \le 3( 3mq  + 1 )\frac{c_1k}{n}
    + \frac{4\gamma^{-2}+8}{3n}.
  $$
\end{lemma}

In addition, in our PAC-Bayesian arguments, the following quantity arises as a smoothing residual:
\begin{align}
\calE_{\gamma,B} := \esp\left[\unit_{{\gamma\Vert\bx\Vert_2\le B}}
\exp\left(
-\frac{B^2}{2\gamma^2\Vert\bx\Vert_2^2}
\right)
\frac{\gamma\Vert\bx\Vert_2}{B}
\right].
\label{def:exponential:residual}
\end{align}
Lemma \ref{lemma:exp:Q} in the following gives a bound on this quantity.

Denote by \(W:[0,\infty)\to[0,\infty)\) the Lambert $W$-function, that is, the inverse of the map 
$
  \phi(t) := te^t
$
on $t\ge0$.
\begin{lemma}[Smoothing residual; proven subsequently]\label{lemma:exp:Q}
Suppose 
\begin{align}
    B \ge \gamma \tr^{1/2}(\bfSigma)
\left(
\kappa_p^p\frac{n}{c_1k}
\right)^{\frac1p}
\sqrt{\min\left\{\sigma_p^2,
    2 W\!\left( \left( \frac{n}{c_1k} \right)^2 \right)
\right\}}.\label{lemma:exp:Q:condition}
\end{align}
Then 
$
  \calE_{\gamma,B} \le 2\sqrt{\frac{\pi}{2}}c_1\frac{k}{n}.
$
\end{lemma}

To prove these lemmas we need with some auxiliary results.

\subsubsection{Auxiliary results \& proof of Lemma \ref{lemma:exp:Q}}
To prove Lemma \ref{lemma:exp:Q}, we need the following proposition.
\begin{proposition}[Smoothed counting probability; proven subsequently]\label{prop:counting:linear:v2}
  Under the assumptions of Lemma \ref{lemma:count:linear:v2}, 
    \[
        \sup_{\bfU\in\mbB_2^{m\times q}}\Gamma_{\bfU,\gamma}\prob\big(
            \|\bfTheta\|_{\infty}(\bx) > B
        \big) \le 
    ( 3mq  + 1 )\frac{c_1k}{n}.
    \]
\end{proposition}

To prove Lemma \ref{lemma:exp:Q} and Proposition \ref{prop:counting:linear:v2}, we need Lemma \ref{lemma:normal:residual:counting:v2}, stated next. Its proof is given in Section \ref{ss:suplement:lemma:normal:residual:counting:v2} of the supplement. 
\begin{lemma}[One-dimensional Gaussian tail; proof in \S \ref{ss:suplement:lemma:normal:residual:counting:v2} of supplement]
\label{lemma:normal:residual:counting:v2}
For any $\bz\in\re^{d}\setminus\{\mathbf{0}\}$, 
\[
  \exp\left(-\frac{B^2}{2\gamma^2\Vert\bz\Vert_2^2}\right)
    \frac{\gamma\Vert\bz\Vert_2}{B+\frac{\gamma^2\Vert\bz\Vert_2^2}{B}}
  \le \frac{\Gamma_{\mathbf{0},\gamma}\unit_{\{|\langle\bz,\btheta\rangle| > B\}}}{\sqrt{2/\pi}}
  \le \exp\left(-\frac{B^2}{2\gamma^2\Vert\bz\Vert_2^2}\right)\frac{\gamma\Vert\bz\Vert_2}{B}.
\]
\end{lemma}

Assuming these results, we can prove Lemma \ref{lemma:exp:Q}. 
\begin{proof}[Proof of Lemma \ref{lemma:exp:Q}]
\textbf{Step 1.} We start proving that
\[
    B \ge \gamma \tr^{1/2}(\bfSigma)
\left(
\kappa_p^p\frac{n}{c_1k}
\right)^{\frac1p}
\sqrt{
  2 W\!\left( \left( \frac{n}{c_1k} \right)^2 \right)
}
\quad \Longrightarrow \quad 
\calE_{\gamma,B}\le2\sqrt{\frac{\pi}{2}} c_1\frac{k}{n}.\label{lemma:exp:Q:eq1}
\]

Define the event
$
\mathsf{E}:=
\left\{
\Vert\bx\Vert_2^2
\le
\left(
1+
\left(
\kappa_p^p\frac{n}{c_1k}
\right)^{\frac{2}{p}}
\right)
\tr(\bfSigma)
\right\}.
$
By Markov's inequality,
\[
\prob(\mathsf{E}^c)
\le
c_1\frac{k}{n}.\label{lemma:exp:Q:eq2}
\]

Let us assume that, for some $\theta>0$, 
$
B
\ge
\theta\cdot
\gamma
\tr^{1/2}(\bfSigma)
\left(
\kappa_p^p\frac{n}{c_1k}
\right)^{\frac1p}.
$
On the event
$\{\gamma\Vert\bx\Vert_2\le B\}\cap\mathsf{E}$,
\begin{align}
\gamma^2\Vert\bx\Vert_2^2
\le
\gamma^2\tr(\bfSigma)
\left(
1+
\left(
\kappa_p^p\frac{n}{c_1k}
\right)^{\frac{2}{p}}
\right)
\le
\frac{B^2}{\theta^2}
\left(
\left(
\kappa_p^p\frac{n}{c_1k}
\right)^{-\frac2p}
+1
\right),
\end{align}
and, in particular,
$
\frac{B^2}{\gamma^2\Vert\bx\Vert_2^2}
\ge
\frac{\theta^2}{
\left(
\left(
\kappa_p^p\frac{n}{c_1k}
\right)^{-\frac2p}
+1
\right)
}.
$

Splitting the expectation defining \(\calE_{\gamma,B}\)
according to the event \(\mathsf{E}\), we obtain
\begin{align}
\calE_{\gamma,B}
&\le
\esp\left[
\unit_{\{\gamma\Vert\bx\Vert_2\le B\}\cap\mathsf{E}}
\exp\left(
-\frac{B^2}{2\gamma^2\Vert\bx\Vert_2^2}
\right)
\frac{\gamma\Vert\bx\Vert_2}{B}
\right]
+
\prob(\mathsf{E}^c)\\
&\le
\exp\left(
-
\frac{\theta^2/2}{
\left(
\left(
\kappa_p^p\frac{n}{c_1k}
\right)^{-\frac2p}
+1
\right)
}
\right)
\frac{
\sqrt{
\left(
\kappa_p^p\frac{n}{c_1k}
\right)^{-\frac2p}
+1
}
}{\theta}
+
c_1\frac{k}{n},\label{lemma:exp:Q:eq2'}
\end{align}
where in the second inequality we used that the function 
$t\mapsto
\exp\left(-t/2\right)/\sqrt{t}
$
is decreasing on $(0,\infty)$ and the bound \eqref{lemma:exp:Q:eq2}.

Let
$
a:=
(
\kappa_p^p\frac{n}{c_1k}
)^{-\frac2p}
$
and choose $\theta>0$ such that
$
\frac{\sqrt{1+a}}{\theta}
\exp\left(
-\frac{\theta^2}{2(1+a)}
\right)
=
c_1\frac{k}{n}.
$
An elementary computation entails 
$
  \theta = \sqrt{
    (1+a)W\left(
      \left(\frac{n}{c_1k}\right)^2
    \right)
  }. 
$
By \eqref{lemma:exp:Q:eq2'}, 
$
\calE_{\gamma,B}
\le
2c_1\frac{k}{n} \le 2\sqrt{\frac{\pi}{2}} c_1\frac{k}{n},
$
and the assumption on $B$ becomes
\[
B
\ge
\gamma
\tr^{1/2}(\bfSigma)
\left(
\kappa_p^p\frac{n}{c_1k}
\right)^{\frac1p}
\sqrt{
\left(
1+a
\right)
W\!\left(
\left(
\frac{n}{c_1k}
\right)^2
\right)
}.
\]
Since $a\le1$, for the displayed condition on $B$ to hold, it is sufficient that 
$
   B \ge \gamma \tr^{1/2}(\bfSigma)
\left(
\kappa_p^p\frac{n}{c_1k}
\right)^{\frac1p}
\sqrt{
  2 W\!\left( \left( \frac{n}{c_1k} \right)^2 \right)
}.
$ 
This proves \eqref{lemma:exp:Q:eq1}.

\textbf{Step 2.} Next, we prove that 
\[
    B \ge \gamma \tr^{1/2}(\bfSigma)
\left(
\sigma_p^p\kappa_p^p\frac{n}{c_1k}
\right)^{\frac1p}
\quad \Longrightarrow \quad 
\calE_{\gamma,B} \le 2\sqrt{\frac{\pi}{2}} c_1\frac{k}{n}.
\label{lemma:exp:Q:eq3}
\]

Conditionally on $\bx$, $\langle\bx,\btheta\rangle\sim \|\bx\|_2N$. From this fact and the independence between $\bx$ and $\btheta$ (Fubini's theorem),
\begin{align}
  \esp\left[
    \Gamma_{\mathbf{0},1}|\langle\bx,\btheta\rangle|^p
  \right]
  = \esp\left[\esp\left[
    \Gamma_{\mathbf{0},1}|\langle\bx,\btheta\rangle|^p \Big| \bx
  \right]\right] 
  = \sigma_p^p \cdot \esp[\Vert\bx\Vert_2^p]. 
\end{align}
Since $\esp[\Vert\bx\Vert_2^p]\le \kappa_p^p\tr^{\frac{p}{2}}(\bfSigma)$, we obtain 
$
  \esp\left[
    \Gamma_{\mathbf{0},1}|\langle\bx,\btheta\rangle|^p
  \right]
  \le \sigma_p^p \kappa_p^p\tr^{\frac{p}{2}}(\bfSigma). 
$
From Markov's inequality on the distribution $\prob$ and Fubini's theorem, 
\begin{align}
    \Gamma_{\mathbf{0},1}\prob\left(
            \gamma|\langle\btheta,\bx\rangle| > B
        \right)
    = \frac{\gamma^p
        \esp[\Gamma_{\mathbf{0},1}|\langle\btheta,\bx\rangle|^{p}]}
        {B^p}
    \le \frac{
      \gamma^p \sigma_p^p \kappa_p^p\tr^{\frac{p}{2}}(\bfSigma)
    }{ B^p } \le c_1\frac{k}{n}, 
\end{align}
where we used the assumption  
$
  B \ge \gamma \tr^{1/2}(\bfSigma)
(
\sigma_p^p\kappa_p^p\frac{n}{c_1k}
)^{\frac1p}. 
$

By the bound above, Fubini's theorem and Lemma \ref{lemma:normal:residual:counting:v2},
\begin{align}
  \esp\left[\exp\left(-\frac{B^2}{2\gamma^2\Vert\bx\Vert_2^2}\right)
    \frac{\gamma\Vert\bx\Vert_2}{B}
    \unit_{\{\gamma\Vert\bx\Vert_2\le B\}}
  \right]
  &\le 2\esp\left[\exp\left(-\frac{B^2}{2\gamma^2\Vert\bx\Vert_2^2}\right)
    \frac{\gamma\Vert\bx\Vert_2}{B+\frac{\gamma^2\Vert\bx\Vert_2^2}{B}}
  \right] \\
  &\le 2\sqrt{\frac{\pi}{2}}\Gamma_{\mathbf{0},1}\prob\left(
            \gamma|\langle\btheta,\bx\rangle| > B
        \right)\\
  &\le 2\sqrt{\frac{\pi}{2}}c_1\frac{k}{n}.
\end{align}
This proves \eqref{lemma:exp:Q:eq3}.

Having proved both implications \eqref{lemma:exp:Q:eq1} and \eqref{lemma:exp:Q:eq3}, the implication stated in the lemma follows immediately. 
\end{proof}

We now prove Proposition \ref{prop:counting:linear:v2}. We need an elementary lemma whose proof is ommited. 
\begin{lemma}\label{lemma:Lambert}
  For all $x\ge e$, $W(x)\le\log x$. 
\end{lemma}

\begin{proof}[Proof of Proposition \ref{prop:counting:linear:v2}]
    For any $\bfU:=\{\bu_{j,\ell}\}_{j\in[m],\,\ell\in[q]}$, 
    \begin{align}
        \Gamma_{\bfU,\gamma}\prob\left(
            \max_{j\in[m],\,\ell\in[q]}|\langle\bx,\btheta_{j,\ell}\rangle| > B
        \right) &\le 
        \prob\left(
            \max_{j\in[m],\,\ell\in[q]}|\langle\bx,\bu_{j,\ell}\rangle| > \frac{B}{2}
        \right)\\
        &+ \Gamma_{\mathbf{0}_{m\times q},\gamma}\prob\left(
            \max_{j\in[m],\,\ell\in[q]}|\langle\bx,\btheta_{j,\ell}\rangle| > \frac{B}{2}
        \right),
      \label{prop:counting:linear:v2:eq1}
    \end{align}
where $\mathbf{0}_{m\times q}$ denotes the collection of $mq$ null vectors in $\re^d$. 

From \eqref{lemma:count:linear:v2:condition}, $B/2 \ge (\nu_p^p\frac{n}{c_1k})^{\nicefrac{1}{p}}$. From this fact, the union bound and Markov's inequality, 
\[
     \prob\left(
            \max_{j\in[m],\,\ell\in[q]}|\langle\bx,\bu_{j,\ell}\rangle| > \frac{B}{2}
        \right)
    \le mq \frac{c_1k}{n}.
    \label{prop:counting:linear:v2:eq2}
\]

Similarly, from \eqref{lemma:count:linear:v2:condition} 
and the fact that 
$
    \sqrt{\min\{\sigma_p^2,
      4 \log( \frac{n}{c_1k} )
    \}} \ge1
$
(since $\sqrt{e}\frac{c_1k}{n}\le1$ by assumption), 
$
B/2 \ge \gamma\tr^{\nicefrac{1}{2}}(\bfSigma)(\kappa_p^p\frac{n}{c_1k})^{\nicefrac{1}{p}}.
$ 
From this fact and Markov's inequality,
\[
     \prob(
            \gamma\Vert \bx\Vert_2 > B/2
        )
    \le  \frac{c_1k}{n}.
    \label{prop:counting:linear:v2:eq3}
\]

From \eqref{lemma:count:linear:v2:condition} and Lemma \ref{lemma:Lambert} (noting that $\sqrt{e}\frac{c_1k}{n}\le1$ by assumption), 
\begin{align}
  B/2 \ge \gamma\tr^{\nicefrac{1}{2}}(\bfSigma)
  \left(
        \kappa_p^p\frac{n}{c_1k}
      \right)^{\frac1p}
  \sqrt{\min\left\{\sigma_p^2,
      2W\left(
      \left(\frac{n}{c_1k}\right)^2
    \right)
    \right\}}.
\end{align}

From this fact and Lemma \ref{lemma:exp:Q}, 
$
  \calE_{\gamma,B/2}
  \le 2\sqrt{\frac{\pi}{2}}c_1\frac{k}{n}.
$
A partition of the expectation, a union bound, Fubini's theorem and Lemma \ref{lemma:normal:residual:counting:v2} yield
\begin{align}
  &\Gamma_{\mathbf{0}_{m\times q},\gamma}\prob\left(
            \max_{j\in[m],\,\ell\in[q]}|\langle\bx,\btheta_{j,\ell}\rangle| > \frac{B}{2}
        \right)\\
  &\le mq\sqrt{\frac{2}{\pi}}\esp\left[\unit_{\{
    \gamma\Vert\bx\Vert_2 \le B/2
  \}}\exp\left(
            -\frac{(B/2)^2}{2\gamma^2\Vert \bx\Vert_2^2}
        \right)\frac{\gamma\Vert \bx\Vert_2}{(B/2)}
  \right]
  + 
  \prob\left(
    \gamma\Vert \bx\Vert_2 > B/2
  \right)\\
  &\le \left(
    2mq  + 1
    \right)\frac{c_1k}{n},
  \label{prop:counting:linear:v2:eq4}
\end{align}
where in the last inequality we used $
  \calE_{\gamma,B/2}
  \le 2\sqrt{\frac{\pi}{2}}c_1\frac{k}{n}
$ and  \eqref{prop:counting:linear:v2:eq3}. 

Combining the bounds \eqref{prop:counting:linear:v2:eq2} and \eqref{prop:counting:linear:v2:eq4} in \eqref{prop:counting:linear:v2:eq1}, we finish the proof. 
\end{proof}

\subsubsection{Proof of Lemma \ref{lemma:count:linear:v2}}
\label{proof:lemma:count:linear:v2}

We are now ready to prove Lemma \ref{lemma:count:linear:v2}. As in the proof of Lemma 4.2 in \cite{2024Oliveira:Rico}, an argument based on the symmetry of a Gaussian distribution implies that
\begin{align}
  \unit_{\{|\langle\bx_i,\bv\rangle| > B\}}\le 2\Gamma_{\mathbf{0},\gamma}\unit_{\{|\langle\bx_i,\bv+\btheta\rangle| > B\}},\quad 
  \forall\bv\in\mbB_2,\forall i\in[n]. 
\end{align}

In particular, given any $i\in[n]$ and $\bfU=\{\bu_{j,\ell}\}_{j\in[m],\,\ell\in[q]}\in(\re^d)^{m\times q}$, let $(j_i,\ell_i)\in 
\argmax_{j\in[m],\ell\in[q]}|\langle\bx_i,\bu_{j,\ell}\rangle|$. Then, if $\bfTheta=\{\btheta_{j,\ell}\}_{j\in[m],\ell\in[q]}$ are iid Gaussian with mean zero and covariance $\gamma^2\bfI_d$, 
\begin{align}
\unit_{\{\max_{j\in[m],\,\ell\in[q]}|\langle\bx_i,\bu_{j,\ell}\rangle| > B\}}
&= \unit_{\{|\langle\bx_i,\bu_{j_i,\ell_i}\rangle| > B\}}\\
& \le 2\Gamma_{\mathbf{0},\gamma}\unit_{\{|\langle\bx_i,\bu_{j_i,\ell_i}+\btheta_{j_i,\ell_i}\rangle| > B\}}\\
&\le 2\Gamma_{\mathbf{0}_{m\times q},\gamma}\unit_{\{\max_{j\in[m],\,\ell\in[q]}|\langle\bx_i,\bu_{j,\ell}+\btheta_{j,\ell}\rangle| > B\}}\\
&= 2\Gamma_{\bfU,\gamma}\unit_{\{\max_{j\in[m],\,\ell\in[q]}|\langle\bx_i,\btheta_{j,\ell}\rangle| > B\}},
\label{lemma:count:linear:v2:eq1}
\end{align}
where $\mathbf{0}_{m\times q}$ denotes the collection of $mq$ null vectors in $\re^d$.

We will apply Proposition  \ref{prop:bernstein:smoothed} with  $X_i(\bfTheta):=\unit_{\{\max_{j\in[m],\,\ell\in[q]}|\langle\bx_i,\btheta_{j,\ell}\rangle|>B\}}$ and $\bar A:=1$. First, Proposition \ref{prop:counting:linear:v2} with $B$ as stated ensures that 
\begin{align}
  \bar\sigma_{\gamma}^2\le \bar\mu_{\gamma} = \sup_{\bfU\in\mbB_2^{m\times q}}\Gamma_{\bfU,\gamma}\esp X_1(\bfTheta) 
    \le ( 3mq  + 1 )\frac{c_1k}{n}.  
  \label{lemma:count:linear:v2:eq2}
\end{align}
Secondly, Proposition \ref{prop:bernstein:smoothed} and 
the relation 
$
  \bar\sigma_\gamma \sqrt{\frac{\gamma^{-2}+2}{n}}
  \le \frac{\bar\sigma_\gamma^2}{2} + \frac{\gamma^{-2}+2}{2n}
$
yield
\begin{align}
\esp\left[\sup_{\bfU\in\mbB_2^{m\times q}}\frac{1}{n}\sum_{i=1}^n
            \Gamma_{\bfU,\gamma}X_i(\bfTheta)\right]
  \le \frac{3}{2}\bar\mu_\gamma
            + \frac{4\gamma^{-2}+8}{6n}.
\end{align}

From the above bound, \eqref{lemma:count:linear:v2:eq1} and \eqref{lemma:count:linear:v2:eq2}, the inequality claimed in Lemma \ref{lemma:count:linear:v2} is proved.

\subsection{Concentration of the counting process}\label{ss:counting:process}

Fix $Q>0$. In this section we establish concentration of the counting process in \eqref{equation:counting:process}. 

\begin{lemma}[Concentration of the counting process; proof below]\label{lemma:count:linear:v3}
  Assume 
  \begin{align}
    \frac{Q^{1/q}}{2m^{\nicefrac{1}{q}}}
    \ge \left(\nu_p^p\frac{n}{c_1k}\right)^{\frac{1}{p}}
    \bigvee 
    \left(
      \gamma \tr^{1/2}(\bfSigma)
      \left(
        \kappa_p^p\frac{n}{c_1k}
      \right)^{\frac1p}
      \sqrt{\min\left\{\sigma_p^2,
      4 \log\left( \frac{n}{c_1k} \right)
    \right\}}
    \right).
  \label{lemma:count:linear:condition:B:c:v3}
  \end{align}
  Let any $\delta\in(0,1)$. Then, with probability at least $1-\delta$,
  \begin{align}
    Z_{(AQ)} \le Z_{Q} \le 
    \frac{14\gamma^{-2}+28+4\log(1/\delta)}{3n}
    + \frac{21}{2}( 3mq  + 1 )c_1\frac{k}{n}.
  \end{align}
\end{lemma}
\begin{proof}
We start claiming that, for any $t>0$,  
\begin{align}
  Z_{t} \leq \sup_{\bfU\in\mbB_2^{m\times q}}
        \frac{1}{n}\sum_{i=1}^n\unit_{
          \left\{
          \|\bfU\|_{\infty}(\bx_i) > 
          \left(\frac{t}{m}\right)^{\nicefrac{1}{q}} 
          \right\}}. 
  \label{lemma:count:linear:v3:eq1}
\end{align}
Indeed, fix $t>0$. Given any $\bfU\in\mbB_2^{m\times q}$ and $i\in[n]$, 
$
\left|
      \sum_{j=1}^{m}
      \prod_{\ell=1}^{q}
      \langle \bx_i,\bu_{j,\ell}\rangle
    \right|\leq m\, \|\bfU\|^q_\infty(\bx_i)
$, 
yielding
\[
\calI_t\left(\sum_{j=1}^{m}
      \prod_{\ell=1}^{q}
      \langle \bx_i,\bu_{j,\ell}\rangle
    \right)\leq
  \unit_{
          \left\{
          \|\bfU\|_{\infty}(\bx_i) > 
          \left(\frac{t}{m}\right)^{\nicefrac{1}{q}} 
          \right\}}.
\]
From this fact and definition of $Z_t$ it follows that \eqref{lemma:count:linear:v3:eq1} holds. 

Next, let any $Q>0$ satisfying \eqref{lemma:count:linear:condition:B:c:v3} and $\delta\in(0,1)$.

From \eqref{lemma:count:linear:condition:B:c:v3}, Lemma \ref{lemma:count:linear:v2} with truncation threshold $B':=\left(\frac{Q}{m}\right)^{1/q}$ and \eqref{lemma:count:linear:v3:eq1}, we obtain 
\begin{align}
  \esp[Z_{Q}] \le 3( 3mq  + 1 )\frac{c_1k}{n}
    + \frac{4\gamma^{-2}+8}{3n}.
    \label{lemma:count:linear:v3:eq2}
\end{align}

We now combine this estimate with Theorem \ref{thm:bousquet}. For every $\bfU\in\mbB_2^{m\times q}$ and $i\in[n]$, define 
  $
    Y_i(\bfU) := \calI_{Q}\left(
          \sum_{j=1}^{m}
          \prod_{\ell=1}^{q}
          \langle\bx_i,\bu_{j,\ell}\rangle
          \right).
  $
We can apply Theorem \ref{thm:bousquet} with\footnote{As usual, the separability of the associated class follows by restricting the parameter $\bfU$ to a countable dense subset of the finite-dimensional Euclidean ball $\mbB_2^{m\times q}$, so that Theorem~\ref{thm:bousquet} applies to the process above.} 
$
  X_i(\bfU) := Y_i(\bfU) - \esp[Y_1(\bfU)],
$
\begin{align}
        \bar\mu :=
        \esp\left[
            \sup_{\bfU\in\mbB_2^{m\times q}}
            \frac{1}{n}\sum_{i=1}^nX_i(\bfU)
        \right]
        \mbox{ and }
        \bar\sigma^2 :=
        \sup_{\bfU\in\mbB_2^{m\times q}}
        \esp\left[\frac{1}{n}\sum_{i=1}^n
        X_i^2(\bfU)\right].
\end{align}
Indeed, for any $\bfU\in\mbB_2^{m\times q}$ and $i\in[n]$, 
$X_i(\bfU)\le1$ and therefore, by Theorem \ref{thm:bousquet}, 
with probability at least $1-\delta$, 
    \begin{align}
        \sup_{\bfU\in\mbB_2^{m\times q}}
        \frac{1}{n}\sum_{i=1}^n X_i(\bfU)
        \le \bar\mu
        +\sqrt{ 2(2\bar\mu + \bar\sigma^2)}
        \sqrt{\frac{\log(1/\delta)}{n}}
        + \frac{\log(1/\delta)}{3n}
        \le 2\bar\mu + \frac{\bar\sigma^2}{2}
        + \frac{4\log(1/\delta)}{3n}.
        \label{lemma:count:linear:v3:eq3}
    \end{align}
In the second inequality, we used the inequality 
$\sqrt{2xa} \le \frac{x}{2} + a$ with $x=2\bar\mu+\bar\sigma^2$ and 
$a=\frac{\log(1/\delta)}{n}$.

In addition, we have that $\bar\mu\le\esp[Z_{Q}]$, 
\begin{align}
  \bar\sigma^2 \le \sup_{\bfU\in\mbB_2^{m\times q}}
        \frac{1}{n}\sum_{i=1}^n \esp\left[Y_i(\bfU)\right] \le \esp[Z_{Q}]
  ~~\mbox{ and }~~
  Z_{Q} \le \sup_{\bfU\in\mbB_2^{m\times q}}
    \frac{1}{n}\sum_{i=1}^n X_i(\bfU)  + \esp[Z_{Q}], 
\end{align}
where we used that $Y_i(\bfU)\in\{0,1\}$. Combining these bounds and \eqref{lemma:count:linear:v3:eq3}, we obtain, with probability at least $1-\delta$,
\begin{align}
  Z_{Q} \le \frac{7}{2}\esp[Z_{Q}] + \frac{4\log(1/\delta)}{3n}.
\end{align}
This and \eqref{lemma:count:linear:v3:eq2} finish the proof.
\end{proof}

Choosing $\gamma^{-2}:=\erank(\bfSigma)$, we obtain the next corollary. Its proof is presented in Section \ref{ss:suplement:cor:count:linear:v3} of the supplementary material. 
\begin{corollary}[of Lemma \ref{lemma:count:linear:v3}; proof in Section \S \ref{ss:suplement:cor:count:linear:v3} of the supplement]
\label{cor:count:linear:v3}
Let $\delta\in(0,5/7)$ and $k\in[n]$ such that 
$
k\ge
(c_2^{-1}\log(1/\delta))
\vee
(c_3^{-1}\erank(\bfSigma)).
$
Assume 
  \begin{align}
    \frac{Q^{1/q}}
    {2m^{\nicefrac{1}{q}}(mq)^{\nicefrac{1}{p}}}
    \ge
    \left(\nu_p^p\frac{n}{c_1k}\right)^{\frac{1}{p}}
    \bigvee 
    \left(\|\bfSigma\|^{\frac{1}{2}}
    \left(
        \kappa_p^p\frac{n}{c_1k}
      \right)^{\frac1p}
      \sqrt{\min\left\{\sigma_p^2,
      4 \log\left( mq \frac{n}{c_1k} \right)
      \right\}}
    \right).
  \label{cor:count:linear:condition:B:c:v3}
  \end{align}
Then with probability at least $1-\delta$,
\begin{align}
Z_{(4^qQ)} \le Z_{Q} \le 
    \left(
      \frac{21}{2}\left( 3  + \frac{1}{mq} \right)c_1 
      + \frac{175}{6}c_2 + \frac{14}{3}c_3
    \right)\frac{k}{n}.
\end{align}
\end{corollary}

\section{Truncation arguments}\label{s:truncation}

Fix $Q>0$. We now establish concentration of the truncated process in \eqref{equation:truncated:process}. The same smoothing strategy is used, but the truncation error is measured in magnitude rather than by an indicator.

For any $s\ge2$, we define 
\begin{align}
  \nu_s (\gamma):= (\nu_s^s + \gamma^s\{\tr(\bfSigma)\}^{s/2}\kappa_{s}^s)^{1/s}.  
  \label{def:truncated:kappa}
\end{align}

\begin{lemma}[Concentration of the truncated process; proof in \S \ref{proof:lemma:truncated:process}]\label{lemma:truncated:process}
Suppose $Q>0$ satisfies \eqref{lemma:count:linear:condition:B:c:v3} in Lemma \ref{lemma:count:linear:v3}. Then, for any $\delta\in(0,1)$, with probability at least $1-\delta$,
\begin{align}
  \epsilon_{(AQ)} &\le 2m^{\frac{r}{2q}}(AQ)^{1-\frac{r}{2q}}
    \Big(2\nu_{r}(\gamma)\Big)^{\frac r2}
    \sqrt{\frac{\gamma^{-2}+2+2\log 2}{n}} 
    + m^{\frac{r}{2q}}(AQ)^{1-\frac{r}{2q}}\nu_{r}^\frac{r}{2}\sqrt{
          \frac{2\log(2/\delta)}{n} 
        } \\
    &+\frac{18\gamma^{-2}+36+4\log 2+8\log(2/\delta)}{3n}(AQ)
    +2\left(
      20mq+10 + \frac{4^q}{A}\sqrt{\frac{\pi}{2}}
    \right)c_1\frac{k}{n}(AQ). 
\end{align}
In above, $\nu_{r}(\gamma)$ is as in \eqref{def:truncated:kappa}.
\end{lemma}

Choosing $\gamma^{-2}:=\erank(\bfSigma)$, we obtain the next corollary of Lemma \ref{lemma:truncated:process}. This corollary is obtained by dividing in the cases $q\le p < 2q$ and $p\ge2q$ and using Young's inequality. See Section \ref{ss:supplement:cor:truncated:process} in the supplement for its proof.
\begin{corollary}[of Lemma \ref{lemma:truncated:process}; proof in Section \S \ref{ss:supplement:cor:truncated:process} in the supplement]\label{cor:truncated:process}
Define 
$
  c_{mq} := 2\left(
        20 + \frac{10+\sqrt{\frac{\pi}{2}}}{mq}
    \right).
$
Let $\delta\in(0,5/7)$ and $k\in[n]$ such that 
$
k\ge
(c_2^{-1}\log(1/\delta))
\vee
(c_3^{-1}\erank(\bfSigma)).
$ 
Let $Q>0$ satisfying \eqref{cor:count:linear:condition:B:c:v3} in Corollary \ref{cor:count:linear:v3}.

  Then, if $p\ge2q$, with probability at least $1-\delta$,
  \begin{align}
    \epsilon_{(4^qQ)} \le 
    3m\left(4\|\bfSigma\|^{\frac{1}{2}}\kappa_{2q}\right)^q
    \sqrt{\frac{\erank(\bfSigma)+11\log(1/\delta)}{n}}
    + \left(
    c_{mq}c_1 + \frac{140}{3}c_2 + 6c_3
  \right)\frac{k}{n}(4^qQ). 
  \end{align}  

  If $q\le p < 2q$, then with probability at least $1-\delta$,
  \begin{align}
    \epsilon_{(4^qQ)} \le 3m\left(4\|\bfSigma\|^{\frac{1}{2}}\kappa_r\right)^{q} \left(
      (11c_2+c_3)\frac{k}{n}
    \right)^{1-\frac{q}{p}}
    +\left(
    c_{mq}c_1 + \frac{239}{3}c_2 + 9c_3
  \right)\frac{k}{n}(4^qQ). 
  \end{align}
\end{corollary}

The proof of Lemma \ref{lemma:truncated:process} will require several steps. In \S \ref{ss:auxiliary:truncation:smoothing} we start with some auxiliary results on truncation and smoothing; the latter will allow us to apply PAC-Bayesian methods. The expectation of the smoothed empirical process is controlled in \S \ref{ss:complexity:smoothed:truncated:process}. This bound is transferred to the (non-smoothed) truncated empirical process in \S \ref{ss:complexity:truncated:process}. Finally, Lemma \ref{lemma:truncated:process} is proven in \S \ref{proof:lemma:truncated:process}.

\subsection{Auxiliary results on truncation and smoothing}
\label{ss:auxiliary:truncation:smoothing}

We start with a lemma giving a bound on the truncation bias.\footnote{Together with Corollary~\ref{cor:truncated:process}, Lemma~\ref{lemma:truncation:bias} shows that the truncation bias \eqref{equation:truncation:bias} is dominated by the fluctuations of the centered truncated process, which are of order $\frac{k}{n}Q$.} 
Its proof is given in Section \ref{ss:supplement:lemma:truncation:bias} in the supplement. 
\begin{lemma}[Truncation bias; proof in section \ref{ss:supplement:lemma:truncation:bias}]\label{lemma:truncation:bias}
\quad 
  \begin{itemize}
  \item[\rm (i)] For any $t>0$, 
  $
  \calT_t \le m^{\frac{p}{q}}\frac{\nu_p^p}{t^{\frac{p}{q}-1}}. 
  $
  \item[\rm (ii)] Let $Q>0$ such that \eqref{cor:count:linear:condition:B:c:v3} in Corollary \ref{cor:count:linear:v3} holds. Then 
  $
    \calT_{(4^qQ)} \le \left(\frac{c_1}{
      2^{3p+\frac{p}{q}}mq
    }\right)
  \frac{k}{n}(4^qQ). 
  $
\end{itemize}
\end{lemma}

Next, we define the smoothed process 
\begin{align}
    \epsilon_{\gamma,t} :=
    \sup_{\bfU\in\mbB_2^{m\times q}}
    \left|
      \Gamma_{\bfU,\gamma}\left(
          \frac{1}{n}\sum_{i=1}^n
          \psi_t\left(
          \sum_{j=1}^{m}
          \prod_{\ell=1}^{q}
          \langle\bx_i,\btheta_{j,\ell}\rangle
          \right)
          -
          \esp\left[
            \psi_t\left(
            \sum_{j=1}^{m}
            \prod_{\ell=1}^{q}
            \langle\bx,\btheta_{j,\ell}\rangle
            \right)
          \right]
      \right)
    \right|,
\end{align}
for properly chosen $t>0$, where integration is with respect to $\bfTheta=\{\btheta_{j,\ell}\}_{j\in[m],\ell\in[q]}$.

The next lemma bounds $\esp[\epsilon_{(AQ)}]$ in terms of $\esp[\epsilon_{\gamma,(AQ)}]$ and additional smoothing residuals. 

\begin{lemma}[Truncation smoothing residual in expectation; proven subsequently]\label{lemma:smoothed:trunc:process:diff}
  \begin{align}
  \esp[\epsilon_{(AQ)}]-\esp[\epsilon_{\gamma,(AQ)}]
  &\le 2AQ\sup_{\bfU\in\mbB_2^{m\times q}}
  \prob\!\left(
      \|\bfU\|_{\infty}^2(\bx) > \frac14
      \left(
        \frac{Q}{m}
      \right)^{2/q}
  \right)\\
  & + 2AQ\esp\!\left[
  \sup_{\bfU\in\mbB_2^{m\times q}}
  \probn_n\!\left(
      \|\bfU\|_{\infty}^2(\bx) > \frac14
      \left(
        \frac{Q}{m}
      \right)^{2/q}
  \right)
  \right] \\
  & +4AQ
  \prob\!\left(
      \gamma^2\|\bx\|_2^2
      >
      \frac14
      \left(
        \frac{Q}{m}
      \right)^{2/q}
  \right) \\
  & + 4^qQ
  \esp\!\left[
  \unit_{\left\{
      \gamma^2\|\bx\|_2^2
      \le
      \frac14
      \left(
        \frac{Q}{m}
      \right)^{2/q}
  \right\}}
  \frac{
      \gamma\|\bx\|_2
  }{
      \left(
        \frac{Q}{m}
      \right)^{1/q}
  }
  \exp\!\left(
      -
      \frac{
      \left(
        \frac{Q}{m}
      \right)^{2/q}
      }{
      8\gamma^2\|\bx\|_2^2
      }
  \right)
  \right].
\end{align}
\end{lemma}

To prove the lemma, we require the following technical result. Its proof, based on elementary properties of the Gaussian distribution, is given in Section \ref{appendix:ss:gaussian:lemmas} of the supplementary material.  

\begin{lemma}[Truncation residual a.s.; proof in Section \S \ref{appendix:ss:gaussian:lemmas} of the supplement]\label{lemma:normal:residual}
Let $\bfU=\{\bu_{j,\ell}\}_{j\in[m],\ell\in[q]}
\in(\re^d)^{m\times q}$ and $\bz\in\re^d$ satisfy
$$
\max_{j\in[m],\,\ell\in[q]}
\langle \bz,\bu_{j,\ell}\rangle^2
\le
\frac14
\left(\frac{Q}{m}\right)^{2/q}
\mbox{ and }
\gamma^2\|\bz\|_2^2
\le
\frac14
\left(\frac{Q}{m}\right)^{2/q}.
$$

Then 
\begin{align}
\Bigg|
\sum_{j=1}^{m}
\prod_{\ell=1}^{q}
\langle \bz,\bu_{j,\ell}\rangle
-
\Gamma_{\bfU,\gamma}
\!\left[
\psi_{AQ}\!\left(
\sum_{j=1}^{m}
\prod_{\ell=1}^{q}
\langle \bz,\btheta_{j,\ell}\rangle
\right)
\right]
\Bigg|
&\le \frac{4^q Q}{2} \cdot \frac{\gamma\|\bz\|_2}
    {\left(
      \frac{Q}{m}
    \right)^{\nicefrac{1}{q}}}
\exp\!\left(
-\frac{
\left(\frac{Q}{m}\right)^{2/q}
}{
8\gamma^2\|\bz\|_2^2
}
\right).
\end{align}
\end{lemma}

\begin{proof}[Proof of Lemma \ref{lemma:smoothed:trunc:process:diff}]
Define 
$
f_Q(\bfU,\bz)
:= \psi_{AQ}\!\left(
\sum_{j=1}^{m}
\prod_{\ell=1}^{q}
\langle\bz,\bu_{j,\ell}\rangle
\right),
$
and 
\begin{align}
\epsilon(\bfU)
&:=
\espn_n\Big(
f_Q(\bfU,\bx)
-
\esp\big[f_Q(\bfU,\bx)\big]
\Big),\\
\epsilon_Q(\bfU)
&:=
\espn_n\!\left(
\Gamma_{\bfU,\gamma}
\Big(
f_Q(\bfTheta,\bx)
-
\esp\big[f_Q(\bfTheta,\bx)\big]
\Big)
\right).
\end{align}
Define also
$
\Delta(\bfU,\bx)
:=
f_Q(\bfU,\bx)
-
\Gamma_{\bfU,\gamma}
\!\left(
f_Q(\bfTheta,\bx)
\right).
$

Since
$
\epsilon(\bfU)
=
\epsilon_Q(\bfU)
+
\espn_n(
\Delta(\bfU,\bx)
-
\esp[\Delta(\bfU,\bx)]
),
$
it follows that
\begin{align}
\esp[\epsilon_{(AQ)}]
&\le
\esp[\epsilon_{\gamma,(AQ)}]
+
\esp\!\left[
\sup_{\bfU\in\mbB_2^{m\times q}}
\espn_n
|\Delta(\bfU,\bx)|
\right]
+
\sup_{\bfU\in\mbB_2^{m\times q}}
\esp
|\Delta(\bfU,\bx)|.
\label{lemma:smoothed:trunc:process:diff:eq1}
\end{align}

We bound the second term in
\eqref{lemma:smoothed:trunc:process:diff:eq1}.
On the event where either
\begin{align}
\gamma^2\Vert\bx\Vert_2^2 > \frac14
\left(
\frac{Q}{m}
\right)^{2/q}
\mbox{ or }
\max_{j\in[m],\,\ell\in[q]}\langle\bx,\bu_{j,\ell}\rangle^2 >
\frac14
\left(
\frac{Q}{m}
\right)^{2/q},
\end{align}
we have
\(
|\Delta(\bfU,\bx)|
\le 2AQ.
\)
On the complementary event, it is easy to check that 
$
  \left|
    \sum_{j=1}^m\prod_{\ell=1}^q\langle\bx,\bu_{j,\ell}\rangle 
  \right| \le 2^{-q}Q \le AQ,
$
for $A\ge1$; hence 
$
  f_Q(\bfU,\bx) = \sum_{j=1}^m\prod_{\ell=1}^q\langle\bx,\bu_{j,\ell}\rangle. 
$
By Lemma \ref{lemma:normal:residual}, on the complementary event,
$
|\Delta(\bfU,\bx)|
\le
\frac{4^qQ}{2}
\left(
\frac{
\gamma\Vert\bx\Vert_2
}{
\left(\frac{Q}{m}\right)^{1/q}
}
\right)
\exp\!\left(
-\frac{
\left(\frac{Q}{m}\right)^{2/q}
}{
8\gamma^2\Vert\bx\Vert_2^2
}
\right).
$

Partitioning the expectation over the two disjoint events,
\begin{align}
\sup_{\bfU\in\mbB_2^{m\times q}}
\esp
|\Delta(\bfU,\bx)|
&\le
2AQ
\sup_{\bfU\in\mbB_2^{m\times q}}
\prob\!\left(
\max_{j\in[m],\,\ell\in[q]}
\langle\bx,\bu_{j,\ell}\rangle^2
>
\frac14
\left(
\frac{Q}{m}
\right)^{2/q}
\right)
\\
&\quad
+
2AQ
\prob\!\left(
\gamma^2\Vert\bx\Vert_2^2
>
\frac14
\left(
\frac{Q}{m}
\right)^{2/q}
\right)
\\
&\quad
+
\frac{4^qQ}{2}
\esp\!\left[
\unit_{\left\{
\gamma^2\Vert\bx\Vert_2^2
\le
\frac14
\left(
\frac{Q}{m}
\right)^{2/q}
\right\}}
\left(
\frac{
\gamma\Vert\bx\Vert_2
}{
\left(\frac{Q}{m}\right)^{1/q}
}
\right)
\exp\!\left(
-\frac{
\left(\frac{Q}{m}\right)^{2/q}
}{
8\gamma^2\Vert\bx\Vert_2^2
}
\right)
\right].
\end{align}

Similarly, we bound the first term in
\eqref{lemma:smoothed:trunc:process:diff:eq1}.
We apply the same argument for the empirical distribution instead of the population distribution and then take the expectation. This yields
\begin{align}
\esp\!\left[
\sup_{\bfU\in\mbB_2^{m\times q}}
\espn_n
|\Delta(\bfU,\bx)|
\right]
&\le
2AQ
\esp\!\left[
\sup_{\bfU\in\mbB_2^{m\times q}}
\probn_n\!\left(
\max_{j\in[m],\,\ell\in[q]}
\langle\bx,\bu_{j,\ell}\rangle^2
>
\frac14
\left(
\frac{Q}{2m}
\right)^{2/q}
\right)
\right]
\\
&\quad
+
2AQ
\prob\!\left(
\gamma^2\Vert\bx\Vert_2^2
>
\frac14
\left(
\frac{Q}{m}
\right)^{2/q}
\right)
\\
&\quad
+
\frac{4^qQ}{2}
\esp\!\left[
\unit_{\left\{
\gamma^2\Vert\bx\Vert_2^2
\le
\frac14
\left(
\frac{Q}{m}
\right)^{2/q}
\right\}}
\left(
\frac{
\gamma\Vert\bx\Vert_2
}{
\left(\frac{Q}{m}\right)^{1/q}
}
\right)
\exp\!\left(
-\frac{
\left(\frac{Q}{m}\right)^{2/q}
}{
8\gamma^2\Vert\bx\Vert_2^2
}
\right)
\right].
\end{align}

To finish, we combine the previous displayed inequalities with
\eqref{lemma:smoothed:trunc:process:diff:eq1}. 
\end{proof}

\subsection{Complexity of the smoothed truncated empirical process}
\label{ss:complexity:smoothed:truncated:process}

Next, we will bound $\esp[\epsilon_{\gamma,Q}]$ using Proposition \ref{prop:bernstein:smoothed}. This is the content of Lemma \ref{lemma:smoothed:trunc:process} in the following. Before, we state a elementary lemma which will be useful to bound the variance parameter. Its proof is given in Section \ref{ss:supplement:lemma:variance:auxiliary} of the supplement. 
\begin{lemma}[Proof in \S \ref{ss:supplement:lemma:variance:auxiliary} of supplement]\label{lemma:variance:auxiliary}
  Let $\{X_{\ell,j}\}_{\ell\in[m],j\in[q]}$ be a finite collection of random variables such that 
  $
    \max_{\ell\in[m],j\in[q]}\esp[|X_{\ell,j}|^{r}]<\infty. 
  $
  Then  
  \begin{align}
    \esp\left[
      \psi_{AQ}^2\left(\sum_{j=1}^m\prod_{\ell=1}^q X_{\ell,j}\right)
    \right] &\le 
    m^{\frac{r}{q}-1} (AQ)^{2-\frac{r}{q}}
  \sum_{j=1}^m\prod_{\ell=1}^q
      \left(\esp\left[
      \left|X_{\ell,j}\right|^{r}
  \right]\right)^{\frac{1}{q}}. 
  \end{align}
\end{lemma}

\begin{lemma}[Smoothed truncated process: complexity]\label{lemma:smoothed:trunc:process}
    \begin{align}
    \esp[\epsilon_{\gamma,(AQ)}]
    \le m^{\frac{r}{2q}}(AQ)^{1-\frac{r}{2q}}\,\Big(2\nu_{r}(\gamma)\Big)^\frac{r}{2}
      \sqrt{\frac{\gamma^{-2}+2(1+\log 2)}{n}}
      + \frac{\gamma^{-2}+2(1+\log 2)}{3n}AQ.
    \end{align}
\end{lemma}
\begin{proof}
We apply Proposition \ref{prop:bernstein:smoothed} with $\bar A:=2AQ$,  $\bar\mu_\gamma = 0$,
\begin{align}
X_i(\bfTheta)
&:=
\psi_{AQ}\left(
\sum_{j=1}^{m}
\prod_{\ell=1}^{q}
\langle\bx_i,\btheta_{j,\ell}\rangle
\right)
-
\esp\left[
\psi_{AQ}\left(
\sum_{j=1}^{m}
\prod_{\ell=1}^{q}
\langle\bx,\btheta_{j,\ell}\rangle
\right)
\right],\\
\bar\sigma_\gamma^2
&:=
\sup_{\bfU\in\mbB_2^{m\times q}}
\Gamma_{\bfU,\gamma}
\var\left(
\psi_{AQ}\left(
\sum_{j=1}^{m}
\prod_{\ell=1}^{q}
\langle\bx,\btheta_{j,\ell}\rangle
\right)
\right).
\end{align}
We obtain 
\begin{align}
    \esp[\epsilon_{\gamma,(AQ)}] &\le
    \bar\sigma_\gamma
    \sqrt{\frac{\gamma^{-2}+2(1+\log 2)}{n}}
    + \frac{\gamma^{-2}+2(1+\log 2)}{6n}2AQ.
\end{align}

It remains next to obtain a bound on the variance parameter $\bar\sigma_\gamma$. By Lemma \ref{lemma:variance:auxiliary}, 
\begin{align}
  \bar\sigma_\gamma^2
  &\le \sup_{\bfU\in\mbB_2^{m\times q}}
  \Gamma_{\mathbf 0,\gamma}
  \esp\left[
  \psi_{AQ}^2\left(
  \sum_{j=1}^{m}
  \prod_{\ell=1}^{q}
  \langle\bx,\bu_{j,\ell}+\gamma\btheta_{j,\ell}\rangle
  \right)
  \right]\\ 
  &\le m^{r/q-1} (AQ)^{2-\frac{r}{q}}
  \sup_{\bfU\in\mbB_2^{m\times q}}
  \sum_{j=1}^{m}
  \Gamma_{\mathbf 0,\gamma}
  \prod_{\ell=1}^{q}
  \left(
  \esp\left[
  |\langle\bx,\bu_{j,\ell}+\gamma\btheta_{j,\ell}\rangle|^{r}
  \right]
  \right)^{1/q}\\
  &= m^{r/q-1} (AQ)^{2-\frac{r}{q}}
  \sup_{\bfU\in\mbB_2^{m\times q}}
  \sum_{j=1}^{m}
  \prod_{\ell=1}^{q}
  \Gamma_{\mathbf 0,\gamma}
  \left(
  \esp\left[
  |\langle\bx,\bu_{j,\ell}+\gamma\btheta_{j,\ell}\rangle|^{r}
  \right]
  \right)^{1/q}\\
  &\le m^{r/q-1} (AQ)^{2-\frac{r}{q}}
  \sup_{\bfU\in\mbB_2^{m\times q}}
  \sum_{j=1}^{m}
  \prod_{\ell=1}^{q}
  \left(
  \Gamma_{\mathbf 0,\gamma}
  \esp\left[
  |\langle\bx,\bu_{j,\ell}+\gamma\btheta_{j,\ell}\rangle|^{r}
  \right]
  \right)^{1/q}.
  \label{lemma:smoothed:trunc:process:eq1}
\end{align}
The equality follows from the product structure of $\Gamma_{\mathbf{0},\gamma}$ while the inequality follows from Jensen. 

For any $\bfU\in\mbB_2^{m\times q}$, using $|a+b|^r\le2^{r-1}(|a|^r+|b|^r)$, $\Gamma_{\mathbf{0},\gamma}|\langle\bx,\gamma\btheta\rangle|^r=\gamma^r\|\bx\|_2^r$, Fubini's theorem and $\esp[\|\bx\|_2^r]\le\kappa_r^r\tr^{\frac{r}{2}}(\bfSigma)$, 
\begin{align}
  \Gamma_{\mathbf 0,\gamma}
  \esp\left[
  |\langle\bx,\bu_{j,\ell}+\gamma\btheta_{j,\ell}\rangle|^{r}
  \right]
  \le
  2^{r-1}
  \left(
  \nu_{r}^{r} + \kappa_{r}^{r}\gamma^{r}\tr^\frac{r}{2}(\bfSigma)
  \right)
  = 2^{r-1}\nu_r^r(\gamma).
  \label{lemma:smoothed:trunc:process:eq2}
\end{align}
Combining the bounds, 
$
  \bar\sigma_\gamma^2 \le \,m^{\frac{r}{q}}(AQ)^{2-\frac{r}{q}}\,2^{r}\nu_r^r(\gamma). 
$

To finish, we use this bound on the bound for $\esp[\epsilon_{\gamma,(AQ)}]$.
\end{proof}

\subsection{Complexity of the truncated empirical process}
\label{ss:complexity:truncated:process}

Markov's inequality together with  Lemmas \ref{lemma:count:linear:v2}, \ref{lemma:exp:Q} and \ref{lemma:Lambert} bound the residual terms in Lemma \ref{lemma:smoothed:trunc:process:diff}. Together with Lemma \ref{lemma:smoothed:trunc:process}, we obtain a bound on $\esp[\epsilon_{(AQ)}]$. This is the content of the next proposition. 

\begin{proposition}[Truncated process: complexity]
\label{prop:truncated:complexity}
Let $Q>0$ satisfying \eqref{lemma:count:linear:condition:B:c:v3} in Lemma \ref{lemma:count:linear:v3}.
Then
\begin{align}
  \esp[\epsilon_{(AQ)}]
  &\le m^{\frac{r}{2q}}(AQ)^{1-\frac{r}{2q}}
\Big(2\nu_{r}(\gamma)\Big)^{\frac r2}
\sqrt{\frac{\gamma^{-2}+2+2\log 2}{n}}
+\frac{9\gamma^{-2}+18+2\log 2}{3n}AQ \\
& +\left(
    20mq+10 + \frac{4^q}{A}\sqrt{\frac{\pi}{2}}
  \right)c_1\frac{k}{n}AQ.
\end{align}
\end{proposition}
\begin{proof} We will apply  the bounds of Lemmas \ref{lemma:smoothed:trunc:process:diff} and \ref{lemma:smoothed:trunc:process} together with Lemma
  \ref{lemma:count:linear:v2} and Lemmas \ref{lemma:exp:Q} and \ref{lemma:Lambert} (using that $\sqrt{e}\frac{c_1k}{n}\le1$). 
  
From Lemma \ref{lemma:count:linear:v2} with truncation threshold $B':=\frac{Q^{1/q}}{2m^{1/q}}$ and \eqref{lemma:count:linear:condition:B:c:v3},
  \begin{align}
    \esp\left[
      \sup_{\bfU\in\mbB_2^{m\times q}}
      \probn_n\left(
      \max_{j\in[m],\,\ell\in[q]}
      \langle\bx,\bu_{j,\ell}\rangle^2 >
      \frac{1}{4}\left(\frac{Q}{m}\right)^{\frac{2}{q}}
      \right)\right]
      \le 3( 3mq  + 1 )\frac{c_1k}{n}
      + \frac{4\gamma^{-2}+8}{3n}. 
    \label{prop:truncated:complexity:eq2}
  \end{align}

Lemmas \ref{lemma:exp:Q} and \ref{lemma:Lambert} with truncation threshold $B':=\frac{Q^{1/q}}{2m^{1/q}}$ and \eqref{lemma:count:linear:condition:B:c:v3} imply
  \begin{align}
  \esp\left[\unit_{
    \left\{
      \gamma^2\Vert\bx\Vert_2^2\le
      \frac{1}{4}\left(\frac{Q}{m}\right)^{\frac{2}{q}}
    \right\}
  }
    \left(\frac{2\gamma\Vert\bx\Vert_2}{
        \left(\frac{Q}{m}\right)^{\nicefrac{1}{q}}
      }\right)
      \exp\left(
            -\frac{
            \left(\frac{Q}{m}\right)^{\nicefrac{2}{q}}
            }{8\gamma^2\Vert\bx\Vert_2^2}
          \right)
  \right]
  \le 2\sqrt{\frac{\pi}{2}}c_1\frac{k}{n}. 
  \label{prop:truncated:complexity:eq1}
  \end{align}

  From \eqref{lemma:count:linear:condition:B:c:v3} and the fact that 
    $
    \sqrt{\min\{\sigma_p^2,
      4 \log( \frac{n}{c_1k} )
    \}} \ge1
      $
(since $\sqrt{e}\frac{c_1k}{n}\le1$ by assumption), we obtain 
  $
    \frac{1}{2}\left(\frac{Q}{m}\right)^{\nicefrac{1}{q}}
    \ge \gamma\tr^{\frac{1}{2}}(\bfSigma)
    (\kappa_p^p\frac{n}{c_1k})^{\nicefrac{1}{p}}.
  $
  From this fact and Markov's inequality and a union bound,
  \begin{align}
    \sup_{\bfU\in\mbB_2^{m\times q}}
    \prob\left(
      \max_{j\in[m],\,\ell\in[q]}
      \langle\bx,\bu_{j,\ell}\rangle^2 >
      \frac{1}{4}\left(\frac{Q}{m}\right)^{\nicefrac{2}{q}}
      \right)
    \le mq\frac{c_1k}{n}.
    \label{prop:truncated:complexity:eq3}
  \end{align}
  Likewise, from 
  $
    \frac{1}{2}\left(\frac{Q}{m}\right)^{\nicefrac{1}{q}}
    \ge \gamma\tr^{\frac{1}{2}}(\bfSigma)
    (\kappa_p^p\frac{n}{c_1k})^{\nicefrac{1}{p}}
  $
  and Markov's inequality, 
  \begin{align}
    \prob\left(
      \gamma^2\Vert \bx\Vert_2^2 >
      \frac{1}{4}\left(\frac{Q}{m}\right)^{\nicefrac{2}{q}}
    \right) \le \frac{c_1k}{n}.
    \label{prop:truncated:complexity:eq4}
  \end{align}

  The claimed inequality follows from combining 
  \eqref{prop:truncated:complexity:eq2}, 
  \eqref{prop:truncated:complexity:eq1}, 
  \eqref{prop:truncated:complexity:eq3} and 
  \eqref{prop:truncated:complexity:eq4} with the bounds of Lemmas 
  \ref{lemma:smoothed:trunc:process:diff} and 
  \ref{lemma:smoothed:trunc:process}.
\end{proof}

\subsection{Proof of Lemma \ref{lemma:truncated:process}}
\label{proof:lemma:truncated:process}

We are now ready to prove Lemma \ref{lemma:truncated:process}.

Define:
    \begin{align}
        X_i^{\pm}(\bfU) &:=\pm\left(
        \psi_{AQ}\!\left(
          \sum_{j=1}^{m}
          \prod_{\ell=1}^{q}
          \langle\bx_i,\bu_{j,\ell}\rangle
        \right)
        -
        \esp\!\left[
          \psi_{AQ}\!\left(
            \sum_{j=1}^{m}
            \prod_{\ell=1}^{q}
            \langle\bx,\bu_{j,\ell}\rangle
          \right)
        \right]\right),\\
        \bar\sigma^2 &:=
        \sup_{\bfU\in\mbB_2^{m\times q}}
        \esp\!\left[
          \frac{1}{n}\sum_{i=1}^{n}
          X_i^2(\bfU)
        \right],\\
        \bar\mu^{\pm} &:=
        \esp\!\left[
          \sup_{\bfU\in\mbB_2^{m\times q}}
            \frac{1}{n}\sum_{i=1}^{n}
            X_i^{\pm}(\bfU)
        \right], 
        \mbox{ and }
        \bar\mu :=
        \esp\!\left[
          \sup_{\bfU\in\mbB_2^{m\times q}}\left|
            \frac{1}{n}\sum_{i=1}^{n}
            X_i(\bfU)\right|
        \right]. 
    \end{align}
    We will apply Theorem \ref{thm:bousquet} for the classes $\mbB_2^{m\times q}\ni\bfU\mapsto\frac{X_i^+(\bfU)}{2AQ}$ and $\mbB_2^{m\times q}\ni\bfU\mapsto\frac{X_i^{-}(\bfU)}{2AQ}$. 

    For any $\bfU\in\mbB_2^{m\times q}$ and $i\in[n]$,
    $X_i^{+}(\bfU)\le 2AQ$ and $X_i^{-}(\bfU)\le 2AQ$. Theorem \ref{thm:bousquet}, taking into account the rescaling factor $2AQ$, $\bar\mu^{\pm}\le\bar\mu$ and an union bound imply that, with probability at least $1-\delta$:
    \begin{align}
        \left|
        \sup_{\bfU\in\mbB_2^{m\times q}}
        \frac1n\sum_{i=1}^{n}X_i(\bfU)
        \right|
        &\le \bar\mu + \sqrt{
          \frac{ 2\left( 4AQ\bar\mu + \bar\sigma^2 \right)\log(2/\delta)}{n} 
        }
        + \frac{\log(2/\delta)}{3n}2AQ \\
        &\le \bar\mu 
        + 2\sqrt{
          \bar\mu\frac{2AQ\log(2/\delta)}{n} 
        }
        + \bar\sigma\sqrt{
          \frac{2\log(2/\delta)}{n} 
        }
        + \frac{\log(2/\delta)}{3n}2AQ\\
        &\le 2\bar\mu 
        + \bar\sigma\sqrt{
          \frac{2\log(2/\delta)}{n} 
        }
        + \frac{8\log(2/\delta)}{3n}AQ, 
    \end{align}
  where in the last line we used $2\sqrt{xy}\le x+y$ with $x:=\bar\mu$ and $y:=\frac{2AQ\log(2/\delta)}{n}$. 

  We claim that
    $
      \bar \sigma\le m^{\frac{r}{2q}}(AQ)^{1-\frac{r}{2q}}\nu_{r}^\frac{r}{2}. 
    $
    Indeed, by Lemma \ref{lemma:variance:auxiliary}, 
    \begin{align}
      \bar\sigma^2
      &\le \sup_{\bfU\in\mbB_2^{m\times q}}
  \esp\left[
  \psi_{AQ}^2\left(
  \sum_{j=1}^{m}
  \prod_{\ell=1}^{q}
  \langle\bx,\bu_{j,\ell}\rangle
  \right)
  \right]\\
  &\le m^{\frac{r}{q}-1} (AQ)^{2-\frac{r}{q}}
  \sup_{\bfU\in\mbB_2^{m\times q}}
  \sum_{j=1}^{m}
  \prod_{\ell=1}^{q}
  \left(
  \esp\left[
  |\langle\bx,\bu_{j,\ell}\rangle|^{r}
  \right]
  \right)^{1/q}\\
  &\le m^{\frac{r}{q}} (AQ)^{2-\frac{r}{q}}\nu_r^r.
  \end{align}

  Using the bound on $\bar\sigma$,
    $
      \bar\mu = \esp[\epsilon_{(AQ)}]
    $
    and Proposition \ref{prop:truncated:complexity} in the previously obtained concentration inequality, we finish the proof.

\section{The robust estimator}\label{s:robust:estimator}

We recall the estimator $\widehat\bfT_k$ defined in \eqref{label:minimax:trimmed:estimator} in the introduction. We start with the following proposition whose proof is given in Section \ref{ss:supplement:prop:minimax:dir:trimmed:mean} of the supplement.

\begin{proposition}[Minimax trimmed-mean estimator]\label{prop:minimax:dir:trimmed:mean}
  Given any $k\in[n]$, $\widehat\bfT_k$ is well-defined and 
  $$
    \| \widehat\bfT_k(\tilde\bx_{1:n}) - \esp[\bx^{\otimes q}]\|\le2\sup_{\bu\in\mbS_2}
    \left| \sfT_k(\tilde\bx_{1:n}|\bu) - \esp[\langle\bx,\bu\rangle^q] \right|.
  $$
\end{proposition}

By Proposition \ref{prop:minimax:dir:trimmed:mean}, it is enough to control the trimmed mean errors uniformly. Under the \emph{counting condition}, the next lemma reduces this to controlling the truncated mean errors. See Claim~1 in \cite{2024Oliveira:Rico} for a related result for one-sided truncation of positive numbers. The proof of Lemma \ref{lemma:trim:trunc} is given in Section \ref{ss:supplement:lemma:trim:trunc} of the supplement.

\begin{lemma}[Counting $\Rightarrow$ Trimming $\approx$ truncation]\label{lemma:trim:trunc}
Fix $k\in [\lfloor(n-1)/2\rfloor]$ and $Q>0$. Let $\bz_{1:n}=\{\bz_1,\ldots,\bz_n\}$ be any data set. Denote by $\Count_{q}(\bz_{1:n},Q,k)$ the event on which the counting condition holds: for all $\bu\in\mbS_2$, 
$
  \#\{i\in[n]:|\langle\bz_i,\bu\rangle|\ge Q^{1/q}\} \le k.
$
Then, on $\Count_{q}(\bz_{1:n},Q,k)$, for any $\bu\in\mbS_2$, 
\begin{align}
  \left|
    \frac{n}{n-2k} \sfT_Q(\bz_{1:n}\mid\bu) - \sfT_k(\bz_{1:n}\mid\bu)
  \right| \le \frac{2kQ}{n-2k},
  ~~\mbox{ and }~~
  \left|\sfT_k(\bz_{1:n}\mid\bu) -\sfT_Q(\bz_{1:n}\mid\bu)\right| \leq \frac{4kQ}{k}.
\end{align}
\end{lemma}

The proof of Theorem \ref{thm:main} will follow from Proposition \ref{prop:minimax:dir:trimmed:mean}, Lemma \ref{lemma:trim:trunc} and Corollaries \ref{cor:count:linear:v3} and \ref{cor:truncated:process} (for $m=1$). 

\subsection{Proof of Theorem \ref{thm:main}}\label{ss:thm:main:proof}

\begin{remark}
  Recall the constants $C_{\probn}>0$ and $c_0\in(0,1)$ in Assumption \ref{assump:trace:estimator}. In the remainder of this section, $C\ge C_{\probn}$ and $c\in(0,c_0)$ denote constants, as stated in Theorem \ref{thm:main}, and $C_0\ge1$ is an absolute constant. The values of these constants may change from one occurrence to another.
\end{remark}

We set $m=1$. Let $q\in\mathbb{N}$, $q\ge2$. Suppose Assumptions \ref{assump:finite:covariance:main}, \ref{assump:contamination:model}  and \ref{assump:trace:estimator} hold and there is $p \ge2\vee q$ such that $\kappa_p<\infty$. Fix  $\delta\in(0,c)$, $n\ge C(\erank(\bfSigma)\vee\log(4/\delta))$ and $\epsilon\le1/C$. We define, respectively, the trimming and truncation parameters:
\begin{align}
    k &:= \left\lceil \max\Big\{A_1\erank(\bfSigma),~A_2\log\left(\nicefrac{4}{\delta}\right),~A_3(\epsilon n)\Big\}\right\rceil,
  ~~\mbox{ and }~~\\
  Q_k &:= \left(C_0q^{\frac{1}{p}}\|\bfSigma\|^{\frac{1}{2}}\kappa_p\right)^q
    \left(\frac{n}{k}\right)^{\frac{q}{p}}\left(\min\left\{
      p, \log\left( \frac{qn}{k} \right)
      \right\}\right)^{\frac{q}{2}}.
  \label{ss:thm:main:proof:k:Qk}
\end{align}
In above, $A_1,A_2,A_3\ge1$ are constants depending only on $C\ge C_{\probn}$. Within this section, their values may also change from one occurrence to another. $(k,Q_k)$ are theoretical, ensured to exist by our assumptions, but not used in our estimator. Indeed, by enlarging $C\ge C_{\probn}$ if necessary, we can assume that $1\le k<\frac{n}{2}$ as required in the trimmed mean \eqref{label:trimmed:mean:tensor:u}.

By Assumption \ref{assump:trace:estimator}, there is estimator $\widehat\erank$ depending only on $(n,\epsilon,\delta)$ such that, on an event $\calE_1$ of probability at least $1-\delta/4$,
  \begin{align}
    \frac{\erank(\bfSigma)}{3}\le\widehat{\erank}(\tilde\bx_1,\ldots,\tilde\bx_n) \le 3\erank(\bfSigma). 
    \label{proof:thm:main:eq0}
  \end{align}

Next, we will invoke the general Corollaries \ref{cor:count:linear:v3} and \ref{cor:truncated:process} to the specific setting of our estimator. From the conditions on the constants $(c_1,c_2,c_3)$, $k$ and the truncation threshold required by these corollaries and Remark \ref{rem:logical:constants}, we must ensure:
\begin{align}
     0 < c_1 < \left(\frac{n}{\sqrt{e}k}\right)\wedge \frac{1}{e},~~ \left(\frac{\log(\nicefrac{4}{\delta})}{c_2}\right)\vee\left(\frac{\erank(\bfSigma)}{c_3}\right)\le k < \left\lfloor\frac{n-1}{2}\right\rfloor,
    ~\mbox{ and \eqref{cor:count:linear:condition:B:c:v3} 
      for $(k,Q_k)$.}
\end{align}
Note that $\nu_p\le\|\bfSigma\|_2^{1/2}\kappa_p$ and $\sigma_p^2\asymp p\gtrsim \log q\gtrsim 1$. By decreasing $c\le c_0$ and enlarging $C\ge C_{\probn}$ and $C_0\ge1$ if necessary, the displayed conditions are satisfied for $(k,Q_k)$ in \eqref{ss:thm:main:proof:k:Qk} and for any $0<c_1<1/e$ and $0<c_2,c_3<1$.

Corollary \ref{cor:count:linear:v3} implies, after decreasing $0<c_1<1/e$ and $0<c_2,c_3<1$ as necessary, that on an event $\calE_2$ of probability at least $1-\delta/4$,
\begin{align}
  \sup_{\bu\in\mbS_2}\sum_{i=1}^n\unit_{\left\{|\langle\bx_i,\bu\rangle|^q > 4^qQ_k\right\}} \le nZ_{(4^qQ_k)} \le \frac{k}{C_0}.\label{proof:thm:main:eq1}
\end{align}
Corollary \ref{cor:truncated:process} and a union bound imply that, on an event $\calE_3$ of probability at least $1-\delta/2$, if $p\ge2q$, 
  \begin{align}
    \epsilon_{(4^qQ_k)} \le (C_0\|\bfSigma\|^{1/2}\kappa_{2q})^q
    \sqrt{\frac{\erank(\bfSigma)\vee\log(4/\delta)}{n}}
    + C_0\frac{k(4^qQ_k)}{n},
    \label{proof:thm:main:eq4}
  \end{align}
  and if $q\le p \le 2q$,
  \begin{align}
    \epsilon_{(4^qQ_k)} \le (C_0\|\bfSigma\|^{1/2}\kappa_p)^{q} \left(
      \frac{k}{n}
    \right)^{1-\frac{q}{p}}
    + C_0\frac{k(4^qQ_k)}{n}. 
    \label{proof:thm:main:eq5}
  \end{align}

The rest of the proof will happen on the event $\calE_1\cap\calE_2\cap\calE_3$ of probability at least $1-\delta$. 

Recall our estimator defined in \eqref{label:minimax:trimmed:estimator:hat}. By enlarging $C\ge C_{\probn}$ if necessary, we can assume $\widehat k < n/2$ and, thus,   
\begin{align}
  \widehat\bfT_*:=\widehat\bfT_{\widehat k} 
  \quad \mbox{ with }\quad 
  \widehat k := \left\lceil \max\Big\{3A_1\widehat\erank,~A_2\log\left(\nicefrac{4}{\delta}\right),~A_3(\epsilon n)\Big\}\right\rceil.
\end{align}
From \eqref{proof:thm:main:eq0}, $\widehat k\ge k \ge A_3(\epsilon n)$. This fact, \eqref{proof:thm:main:eq1} and Assumption \ref{assump:contamination:model} imply, after we enlarge $A_3\ge1$ and $C\ge C_{\probn}$ if necessary, 
\begin{align}
  \sup_{\bu\in\mbS_2}\sum_{i=1}^n\unit_{\left\{|\langle\tilde\bx_i,\bu\rangle|^q > 4^qQ_k\right\}}
  \le \sup_{\bu\in\mbS_2}\sum_{i=1}^n\unit_{\left\{|\langle\bx_i,\bu\rangle|^q > 4^qQ_k\right\}} + \epsilon n 
  \le \frac{k}{C_0} + \frac{k}{A_3} 
  \le \frac{k}{C_0}
  \le \frac{\widehat k}{C_0}
  \le \widehat k.
\end{align}
Thus, $\Count_{q}(\tilde\bx_{1:n},4^qQ_k,\widehat k)$ holds. 

Proposition \ref{prop:minimax:dir:trimmed:mean} and the triangle inequality imply
\begin{align}
  \| \widehat\bfT_*(\tilde\bx_{1:n}) - \esp[\bx^{\otimes q}]\|
  &\le 2\sup_{\bu\in\mbS_2}
    \left| \sfT_{\widehat k}(\tilde\bx_{1:n}|\bu) 
    - \esp[\langle\bx,\bu\rangle^q] 
    \right|\\
  &\le 2\sup_{\bu\in\mbS_2}
    \left| 
      \sfT_{\widehat k}(\tilde\bx_{1:n}|\bu) - \sfT_{(4^qQ_k)}(\tilde\bx_{1:n}|\bu)
    \right| \\
  &+ 2\sup_{\bu\in\mbS_2}
    \left| 
      \sfT_{(4^qQ_k)}(\tilde\bx_{1:n}|\bu) 
      - \esp[\langle\bx,\bu\rangle^q]
    \right|.
    \label{proof:thm:main:eq2}
\end{align}

Lemma \ref{lemma:trim:trunc}, with the fact that $\Count_{q}(\tilde\bx_{1:n},4^qQ_k,\widehat k)$ holds, implies
\begin{align}
  2\sup_{\bu\in\mbS_2}
    \left| 
      \sfT_{\widehat k}(\tilde\bx_{1:n}|\bu) - \sfT_{(4^qQ_k)}(\tilde\bx_{1:n}|\bu)
    \right|
  &\le \frac{8kQ_k}{n}.
  \label{proof:thm:main:eq6}
\end{align}

From Lemma \ref{lemma:truncation:bias}(ii), noting that condition \eqref{cor:count:linear:condition:B:c:v3} holds for our choice of $(k,Q_k)$, imply 
$
  \calT_{(4^qQ_k)} \le C_0\frac{k(4^qQ_k)}{n}.
$
This fact, Assumption \ref{assump:contamination:model}, $k \ge A_3(\epsilon n)$ and the triangle inequality imply
\begin{align}
  2\sup_{\bu\in\mbS_2}
    \left| 
      \sfT_{(4^qQ_k)}(\tilde\bx_{1:n}|\bu) - \esp[\langle\bx,\bu\rangle^q] 
    \right|
  &\le 2\sup_{\bu\in\mbS_2}
    \left| 
      \sfT_{(4^qQ_k)}(\bx_{1:n}|\bu) - \esp[\langle\bx,\bu\rangle^q] 
    \right| \\
  & + 2\frac{2(\epsilon n)(4^qQ_k)}{n} \\
  &\le 2\epsilon_{(4^qQ_k)} + 2\calT_{(4^qQ_k)} + C_0\frac{k(4^qQ_k)}{n}\\
  &\le 2\epsilon_{(4^qQ_k)} + C_0\frac{k(4^qQ_k)}{n}.
  \label{proof:thm:main:eq3}
\end{align}

We conclude from \eqref{proof:thm:main:eq4}, \eqref{proof:thm:main:eq5}, \eqref{proof:thm:main:eq2}, \eqref{proof:thm:main:eq6} and \eqref{proof:thm:main:eq3} that 
\begin{align}
  \| \widehat\bfT_k(\tilde\bx_{1:n}) - \esp[\bx^{\otimes q}]\|
  &\le \Delta_k + C_0\frac{k(4^qQ_k)}{n},
  \label{proof:thm:main:eq5'}
\end{align}
where 
\begin{align}
  \Delta_k := \begin{cases}
    (C_0\|\bfSigma\|^{1/2}\kappa_{2q})^q
    \sqrt{\frac{\erank(\bfSigma)\vee \log(4/\delta)}{n}}, & \mbox{ if } p \ge 2q,\\
    (C_0\|\bfSigma\|^{1/2}\kappa_p)^{q} \left(
      \frac{k}{n}
    \right)^{1-\frac{q}{p}}, & \mbox{ if } q\le p < 2q.
    \label{proof:thm:main:eq5''}
  \end{cases}
\end{align}

By definition of $Q_k$ in \eqref{ss:thm:main:proof:k:Qk}, 
\begin{align}
  \frac{k(4^qQ_k)}{n}
  \le \left(C_0q^{\frac{1}{p}}\|\bfSigma\|^{\frac{1}{2}}\kappa_p\right)^q \left(
      \frac{k}{n}
    \right)^{1-\frac{q}{p}}
    \left(\min\left\{
      p, \log\left( \frac{qn}{k} \right)
      \right\}\right)^{\frac{q}{2}}.
\end{align}

Suppose first $A_3(\epsilon n)\ge (A_1\erank(\bfSigma))\vee(A_2\log(4/\delta))$; then $k\asymp A_3(\epsilon n)$, entailing
\begin{align}
  \frac{k(4^qQ_k)}{n} &\lesssim  
  \left(C_0q^{\frac{1}{p}}\|\bfSigma\|^{\frac{1}{2}}\kappa_p\right)^q\epsilon^{1-\frac{q}{p}}
  \left(\min\left\{
      p, \log\left(\frac{q}{\epsilon} \right)
      \right\}\right)^{\frac{q}{2}}
  \mbox{ and},\\
  \left(C_0q^{\frac{1}{p}}\|\bfSigma\|^{\frac{1}{2}}\kappa_p\right)^q \left(
      \frac{k}{n}
    \right)^{1-\frac{q}{p}}
  &\lesssim \left(C_0q^{\frac{1}{p}}\|\bfSigma\|^{\frac{1}{2}}\kappa_p\right)^q\epsilon^{1-\frac{q}{p}}.
  \label{proof:thm:main:eq7}
\end{align}
Suppose now that $A_3(\epsilon n)\le (A_1\erank(\bfSigma))\vee(A_2\log(4/\delta))$; then $k \asymp (A_1\erank(\bfSigma))\vee(A_2\log(4/\delta))$ entailing 
\begin{align}
  \frac{k(4^qQ_k)}{n} \lesssim  
  \left(C_0q^{\frac{1}{p}}\|\bfSigma\|^{\frac{1}{2}}\kappa_p\right)^q\left(
    \frac{\erank(\bfSigma)\vee \log(4/\delta)}{n}
  \right)^{1-\frac{q}{p}}
  \left(\min\left\{
      p, \log\left(\frac{qn}{\erank(\bfSigma)\vee\log(4/\delta)} \right)
      \right\}\right)^{\frac{q}{2}}, 
  \label{proof:thm:main:eq8}
\end{align}
and 
\begin{align}
  \left(C_0q^{\frac{1}{p}}\|\bfSigma\|^{\frac{1}{2}}\kappa_p\right)^q \left(
      \frac{k}{n}
    \right)^{1-\frac{q}{p}}
  \lesssim \left(C_0q^{\frac{1}{p}}\|\bfSigma\|^{\frac{1}{2}}\kappa_p\right)^q\left(
    \frac{\erank(\bfSigma)\vee\log(4/\delta)}{n}
  \right)^{1-\frac{q}{p}}.
  \label{proof:thm:main:eq9}
\end{align}

When $p\ge2q$ it suffices to choose $p = 2q$. This implies that the displayed bounds in \eqref{proof:thm:main:eq8} and \eqref{proof:thm:main:eq9} are at most 
\begin{align}
  C_q\|\bfSigma\|^{\frac{q}{2}}\kappa_{2q}^{q}\left(
    \frac{\erank(\bfSigma)\vee \log(4/\delta)}{n}
  \right)^{\frac{1}{2}},
  \label{proof:thm:main:eq10}
\end{align}
for a constant $C_q>0$ that depends only on $q$.

When $q \le p < 2q$, the displayed bounds in \eqref{proof:thm:main:eq8} and \eqref{proof:thm:main:eq9} are at most 
\begin{align}
  C_q\|\bfSigma\|^{\frac{q}{2}}\kappa_{2q}^{q}\left(
    \frac{\erank(\bfSigma)\vee\log(4/\delta)}{n}
  \right)^{1-\frac{q}{p}},
  \label{proof:thm:main:eq11}
\end{align}
for a constant $C_q>0$ that depends only on $q$.

Using \eqref{proof:thm:main:eq7}, \eqref{proof:thm:main:eq8}, \eqref{proof:thm:main:eq9}, \eqref{proof:thm:main:eq10} and \eqref{proof:thm:main:eq11} in \eqref{proof:thm:main:eq5'}-\eqref{proof:thm:main:eq5''}, we finish the proof.

\bibliographystyle{plain} 
\bibliography{robust_simple_tensor_estimation_references}

@article{2021Lim, 
title={Tensors in computations}, 
volume={30}, 
DOI={10.1017/S0962492921000076}, 
journal={Acta Numerica}, 
author={Lim, Lek-Heng}, 
year={2021}, 
pages={555–764}
}

@book{2025Ballard:Kolda, 
place={Cambridge}, 
title={Tensor Decompositions for Data Science}, 
publisher={Cambridge University Press}, 
author={Ballard, Grey and Kolda, Tamara G.}, 
year={2025}}

@article{2021Bi:Tang:Yuan:Zhang:Qu,
   author = "Bi, Xuan and Tang, Xiwei and Yuan, Yubai and Zhang, Yanqing and Qu, Annie",
   title = "Tensors in Statistics", 
   journal= "Annual Review of Statistics and Its Application",
   year = "2021",
   volume = "8",
   number = "Volume 8, 2021",
   pages = "345-368",
   doi = "https://doi.org/10.1146/annurev-statistics-042720-020816"
  }

@article{2025Auddy:Xia:Yuan,
   author = "Auddy, Arnab and Xia, Dong and Yuan, Ming",
   title = "Tensors in High-Dimensional Data Analysis: Methodological Opportunities and Theoretical Challenges", 
   journal= "Annual Review of Statistics and Its Application",
   year = "2025",
   volume = "12",
   number = "Volume 12, 2025",
   pages = "527-551",
   doi = "https://doi.org/10.1146/annurev-statistics-112723-034548"
  }

@article{2007Guedon:Rudelson,
title = {Lp-moments of random vectors via majorizing measures},
journal = {Advances in Mathematics},
volume = {208},
number = {2},
pages = {798-823},
year = {2007},
doi = {https://doi.org/10.1016/j.aim.2006.03.013},
author = {Olivier Guédon and Mark Rudelson}
}

@article{2008Mendelson,
  author  = {Shahar Mendelson},
  title   = {On Weakly Bounded Empirical Processes},
  journal = {Mathematische Annalen},
  volume  = {340},
  number  = {2},
  pages   = {293--314},
  year    = {2008},
  doi     = {10.1007/s00208-007-0149-3},
}

@article{2011Vershynin,
author = {Roman Vershynin},
title = {{Approximating the moments of marginals of high-dimensional distributions}},
volume = {39},
journal = {The Annals of Probability},
number = {4},
pages = {1591 -- 1606},
year = {2011},
doi = {10.1214/10-AOP589}
}

@article{2020Vershynin,
author = {Roman Vershynin},
title = {Concentration inequalities for random tensors},
volume = {26},
journal = {Bernoulli},
number = {4},
pages = {3139 -- 3162},
year = {2020},
doi = {10.3150/20-BEJ1218}
}

@InProceedings{2021Even:Massoulie,
  title = 	 {Concentration of Non-Isotropic Random Tensors with Applications to Learning and Empirical Risk Minimization},
  author =       {Even, Mathieu and Massoulie, Laurent},
  booktitle = 	 {Proceedings of Thirty Fourth Conference on Learning Theory},
  pages = 	 {1847--1886},
  year = 	 {2021},
  volume = 	 {134},
  series = 	 {Proceedings of Machine Learning Research},
  month = 	 {15--19 Aug},
  publisher =    {PMLR},
  url = 	 {https://proceedings.mlr.press/v134/even21a.html}
}

@article{2021Mendelson,
title = {Approximating Lp unit balls via random sampling},
journal = {Advances in Mathematics},
volume = {386},
pages = {107829},
year = {2021},
doi = {https://doi.org/10.1016/j.aim.2021.107829},
author = {Shahar Mendelson}
}

@article{2021Gotze:Holger:Sambale:Sinulis,
author = {Friedrich G{\"o}tze and Holger Sambale and Arthur Sinulis},
title = {{Concentration inequalities for polynomials in $\alpha$-sub-exponential random variables}},
volume = {26},
journal = {Electronic Journal of Probability},
number = {none},
pages = {1 -- 22},
year = {2021},
doi = {10.1214/21-EJP606}
}

@InProceedings{2023Sambale,
author="Sambale, Holger",
title="Some Notes on Concentration for $\alpha$-Subexponential Random Variables",
booktitle="High Dimensional Probability IX",
year="2023",
publisher="Springer International Publishing",
pages="167--192",
doi="10.1007/978-3-031-26979-0_7"
}

@article{2024Zhivotovskiy,
author = {Nikita Zhivotovskiy},
title = {{Dimension-free bounds for sums of independent matrices and simple tensors via the variational principle}},
volume = {29},
journal = {Electronic Journal of Probability},
number = {none},
publisher = {Institute of Mathematical Statistics and Bernoulli Society},
pages = {1 -- 28},
year = {2024},
doi = {10.1214/23-EJP1021}
}

@article{2025Al-Ghattas:Chen:Sanz-Alonso:SharpConcentration,
    author = {Al-Ghattas, Omar and Chen, Jiaheng and Sanz-Alonso, Daniel},
    title = {Sharp concentration of simple random tensors},
    journal = {Information and Inference: A Journal of the IMA},
    volume = {14},
    number = {4},
    pages = {iaaf029},
    year = {2025},
    doi = {10.1093/imaiai/iaaf029}
}

@article{2026Chen:Sanz-Alonso:SharpConcentrationAsymmetry,
    author = {Chen, Jiaheng and Sanz-Alonso, Daniel},
    title = {Sharp concentration of simple random tensors II: asymmetry},
    journal = {Information and Inference: A Journal of the IMA},
    volume = {15},
    number = {2},
    pages = {iaag010},
    year = {2026},
    doi = {10.1093/imaiai/iaag010}
}

@article{2025Al-Ghattas:Chen:Sanz-Alonso,
  author  = {Omar Al-Ghattas and Jiaheng Chen and Daniel Sanz-Alonso},
  title   = {On the estimation of Gaussian moment tensors},
  journal = {Electronic Communications in Probability},
  year    = {2025},
  volume  = {30},
  number  = {none},
  pages   = {1--15},
  doi     = {10.1214/25-ECP734}
}

@article{2025Abdalla:Vershynin,
    author = {Abdalla, Pedro and Vershynin, Roman},
    doi = {10.1007/s10959-025-01458-1},
    journal = {Journal of Theoretical Probability},
	number = {1},
	pages = {3},
	title = {On the Dimension-Free Concentration of Simple Tensors via Matrix Deviation},
	volume = {39},
	year = {2025}
}

@article{2026Bartl:Mendelson,
title = {Uniform mean estimation via generic chaining},
journal = {Advances in Mathematics},
volume = {493},
pages = {110918},
year = {2026},
issn = {0001-8708},
doi = {https://doi.org/10.1016/j.aim.2026.110918},
author = {Daniel Bartl and Shahar Mendelson}
}

@article{2025Shang,
    author = {Zong Shang},
    doi = {https://arxiv.org/abs/2511.06338},
    journal = {Arxiv preprint},
	title = {Upper bounds for the $L^q$ empirical process via generic chaining},
	year = {2025}
}

@article{2018Chen:Gao:Ren,
author = {Mengjie Chen and Chao Gao and Zhao Ren},
title = {{Robust covariance and scatter matrix estimation under Huber’s contamination model}},
volume = {46},
journal = {The Annals of Statistics},
number = {5},
publisher = {Institute of Mathematical Statistics},
pages = {1932 -- 1960},
year = {2018},
doi = {10.1214/17-AOS1607}
}

@article{2018Minsker,
author = {Stanislav Minsker},
title = {{Sub-Gaussian estimators of the mean of a random matrix with heavy-tailed entries}},
volume = {46},
journal = {The Annals of Statistics},
number = {6A},
pages = {2871 -- 2903},
year = {2018},
doi = {10.1214/17-AOS1642},
URL = {https://doi.org/10.1214/17-AOS1642}
}

@article{2020Minsker:Wei,
author = {Stanislav Minsker and Xiaohan Wei},
title = {{Robust modifications of U-statistics and applications to covariance estimation problems}},
volume = {26},
journal = {Bernoulli},
number = {1},
pages = {694 -- 727},
year = {2020},
doi = {10.3150/19-BEJ1149}
}

@InProceedings{2019Ostrovskii:Rudi,
  title = {Affine Invariant Covariance Estimation for Heavy-Tailed Distributions},
  author =  {Ostrovskii, Dmitrii M. and Rudi, Alessandro},
  booktitle = {Proceedings of the Thirty-Second Conference on Learning Theory},
  pages = 	 {2531--2550},
  year = 	 {2019},
  volume = 	 {99},
  series = 	 {Proceedings of Machine Learning Research},
  publisher =    {PMLR}
}

@article{2018Catoni:Giulini,
  author        = {Olivier Catoni and Ilaria Giulini},
  title         = {Dimension-free {PAC}-{B}ayesian Bounds for the Estimation of the Mean of a Random Vector},
  journal       = {arXiv preprint arXiv:1802.04308},
  year          = {2018},
  archivePrefix = {arXiv},
  eprint        = {1802.04308},
  primaryClass  = {math.ST},
  doi           = {10.48550/arXiv.1802.04308},
}

@article{2018Giulini,
  author  = {Ilaria Giulini},
  title   = {Robust Dimension-Free Gram Operator Estimates},
  journal = {Bernoulli},
  volume  = {24},
  number  = {4B},
  pages   = {3864--3923},
  year    = {2018},
  doi     = {10.3150/17-BEJ972},
}

@article{2018Tikhomirov,
    author = {Tikhomirov, Konstantin},
    title = {Sample Covariance Matrices of Heavy-Tailed Distributions},
    journal = {International Mathematics Research Notices},
    volume = {2018},
    number = {20},
    pages = {6254-6289},
    year = {2018},
    doi = {10.1093/imrn/rnx067}
}

@article{2020Mendelson:Zhivotovskiy,
author = {Shahar Mendelson and Nikita Zhivotovskiy},
title = {Robust covariance estimation under $L_{4}-L_{2}$ norm equivalence},
volume = {48},
journal = {The Annals of Statistics},
number = {3},
pages = {1648 -- 1664},
year = {2020},
doi = {10.1214/19-AOS1862}
}

@article{2026Abdalla:Zhivotovskiy,
  author  = {Pedro Abdalla and Nikita Zhivotovskiy},
  title   = {Covariance Estimation: Optimal Dimension-Free Guarantees for Adversarial Corruption and Heavy Tails},
  journal = {Journal of the European Mathematical Society},
  year    = {2026},
  volume  = {28},
  number  = {4},
  pages   = {1809--1847},
  doi     = {10.4171/JEMS/1505}
}

@article{2024Oliveira:Rico,
author = {Roberto I. Oliveira and Zoraida F. Rico},
title = {Improved covariance estimation: Optimal robustness and sub-Gaussian guarantees under heavy tails},
volume = {52},
journal = {The Annals of Statistics},
number = {5},
pages = {1953 -- 1977},
year = {2024},
doi = {10.1214/24-AOS2407}
}

@article{2021Lugosi:Mendelson,
author = {G{\'a}bor Lugosi and Shahar Mendelson},
title = {Robust multivariate mean estimation: The optimality of trimmed mean},
volume = {49},
journal = {The Annals of Statistics},
number = {1},
pages = {393 -- 410},
year = {2021},
doi = {10.1214/20-AOS1961}
}

@article{1999Rudelson,
title = {Random Vectors in the Isotropic Position},
journal = {Journal of Functional Analysis},
volume = {164},
number = {1},
pages = {60-72},
year = {1999},
issn = {0022-1236},
doi = {https://doi.org/10.1006/jfan.1998.3384},
author = {M. Rudelson}
}

@article{2005Klartag:Mendelson,
title = {Empirical processes and random projections},
author = {B. Klartag and S. Mendelson},
journal = {Journal of Functional Analysis},
volume = {225},
number = {1},
pages = {229-245},
year = {2005},
doi = {https://doi.org/10.1016/j.jfa.2004.10.009}
}

@article{2007MendelsonPajorTomczak-Jaegermann,
  author  = {Shahar Mendelson and Alain Pajor and Nicole Tomczak-Jaegermann},
  title   = {Reconstruction and Subgaussian Operators in Asymptotic Geometric Analysis},
  journal = {Geometric and Functional Analysis},
  volume  = {17},
  number  = {4},
  pages   = {1248--1282},
  year    = {2007},
  doi     = {10.1007/s00039-007-0618-7},
}

@article{2010Mendelson,
	author = {Mendelson, Shahar},
	doi = {10.1007/s00039-010-0084-5},
	journal = {Geometric and Functional Analysis},
	number = {4},
	pages = {988--1027},
	title = {Empirical Processes with a Bounded $\Psi_1$ Diameter},
	url = {https://doi.org/10.1007/s00039-010-0084-5},
	volume = {20},
	year = {2010}
}

@article{2010Adamczak:Litvak:Pajor:Tomczak-Jaegermann,
  author  = {Rados{\l}aw Adamczak and Alexander E. Litvak and Alain Pajor and Nicole Tomczak-Jaegermann},
  title   = {Quantitative Estimates of the Convergence of the Empirical Covariance Matrix in Log-Concave Ensembles},
  journal = {Journal of the American Mathematical Society},
  volume  = {23},
  number  = {2},
  pages   = {535--561},
  year    = {2010},
  doi     = {10.1090/S0894-0347-09-00650-X},
}

@article{2012Mendelson:Paouris,
  author  = {Shahar Mendelson and Grigoris Paouris},
  title   = {On Generic Chaining and the Smallest Singular Value of Random Matrices with Heavy Tails},
  journal = {Journal of Functional Analysis},
  volume  = {262},
  number  = {9},
  pages   = {3775--3811},
  year    = {2012},
  doi     = {10.1016/j.jfa.2012.01.025},
}

@article{2012Vershynin,
	author = {Vershynin, Roman},
	doi = {10.1007/s10959-010-0338-z},
	journal = {Journal of Theoretical Probability},
	number = {3},
	pages = {655--686},
	title = {How Close is the Sample Covariance Matrix to the Actual Covariance Matrix?},
	volume = {25},
	year = {2012},
  }

@article{2013Srivastava:Vershynin,
author = {Nikhil Srivastava and Roman Vershynin},
title = {{Covariance estimation for distributions with ${2+\varepsilon}$ moments}},
volume = {41},
journal = {The Annals of Probability},
number = {5},
pages = {3081 -- 3111},
year = {2013},
doi = {10.1214/12-AOP760},
}

@article{2014Mendelson:Paouris,
author = {Shahar Mendelson and Grigoris Paouris},
title = {{On the singular values of random matrices}},
volume = {16},
journal = {J. Eur. Math. Soc.},
number = {4},
publisher = {Bernoulli Society for Mathematical Statistics and Probability},
pages = {823 -- 834},
year = {2014},
doi = {10.4171/JEMS/448}
}

@article{2014Bednorz,
  author        = {Witold Bednorz},
  title         = {Concentration via Chaining Method and Its Applications},
  journal       = {arXiv preprint arXiv:1405.0676},
  year          = {2014},
  archivePrefix = {arXiv},
  eprint        = {1405.0676},
  primaryClass  = {math.PR},
  doi           = {10.48550/arXiv.1405.0676},
}

@article{2015Dirksen,
  author  = {Sjoerd Dirksen},
  title   = {Tail Bounds via Generic Chaining},
  journal = {Electronic Journal of Probability},
  volume  = {20},
  pages   = {1--29},
  year    = {2015},
  doi     = {10.1214/EJP.v20-3760},
}

@article{2016Mendelson,
  author  = {Shahar Mendelson},
  title   = {Upper bounds on product and multiplier empirical processes},
  journal = {Stochastic Processes and their Applications},
  volume  = {126},
  number = {12},
  pages   = {3652--3680},
  year    = {2016},
  doi     = {?},
}

@Inbook{2017Mendelson,
author="Mendelson, Shahar",
title="On Multiplier Processes Under Weak Moment Assumptions",
bookTitle="Geometric Aspects of Functional Analysis: Israel Seminar (GAFA) 2014--2016",
year="2017",
publisher="Springer International Publishing",
pages="301--318",
isbn="978-3-319-45282-1",
doi="10.1007/978-3-319-45282-1_19"
}

@article{2014Lounici,
author = {Karim Lounici},
title = {{High-dimensional covariance matrix estimation with missing observations}},
volume = {20},
journal = {Bernoulli},
number = {3},
publisher = {Bernoulli Society for Mathematical Statistics and Probability},
pages = {1029 -- 1058},
year = {2014},
doi = {10.3150/12-BEJ487},
}

@article{2017Koltchinskii:Lounici,
author = {Vladimir Koltchinskii and Karim Lounici},
title = {{Concentration inequalities and moment bounds for sample covariance operators}},
volume = {23},
journal = {Bernoulli},
number = {1},
pages = {110 -- 133},
year = {2017},
doi = {10.3150/15-BEJ730}
}

@Inbook{2017Liaw:Mehrabian:Plan:Vershynin,
author="Liaw, Christopher
and Mehrabian, Abbas
and Plan, Yaniv
and Vershynin, Roman",
editor="Klartag, Bo'az
and Milman, Emanuel",
title="A Simple Tool for Bounding the Deviation of Random Matrices on Geometric Sets",
bookTitle="Geometric Aspects of Functional Analysis: Israel Seminar (GAFA) 2014--2016",
year="2017",
publisher="Springer International Publishing",
address="Cham",
pages="277--299",
doi="10.1007/978-3-319-45282-1_18"
}

@incollection{2017vanHandel,
  author    = {Ramon van Handel},
  title     = {Structured Random Matrices},
  booktitle = {Convexity and Concentration},
  editor    = {Eric Carlen and Mokshay Madiman and Elisabeth Werner},
  series    = {The IMA Volumes in Mathematics and its Applications},
  volume    = {161},
  pages     = {107--156},
  publisher = {Springer},
  address   = {New York},
  year      = {2017},
  doi       = {10.1007/978-1-4939-7005-6_4},
}

@article{2022Jeong:Li:Plan:Yilmaz,
author = {Jeong, Halyun and Li, Xiaowei and Plan, Yaniv and Yilmaz, Ozgur},
title = {Sub-Gaussian Matrices on Sets: Optimal Tail Dependence and Applications},
journal = {Communications on Pure and Applied Mathematics},
volume = {75},
number = {8},
pages = {1713-1754},
doi = {https://doi.org/10.1002/cpa.22024},
year = {2022}
}

@book{2013Boucheron:Lugosi:Massart,
    author = {Boucheron, Stéphane and Lugosi, Gábor and Massart, Pascal},
    title = {Concentration Inequalities: A Nonasymptotic Theory of Independence},
    publisher = {Oxford University Press},
    year = {2013},
    doi = {10.1093/acprof:oso/9780199535255.001.0001}
}

@inproceedings{2020Diakonikolas:Kane:Pensia,
 author = {Diakonikolas, Ilias and Kane, Daniel M. and Pensia, Ankit},
 booktitle = {Advances in Neural Information Processing Systems},
 pages = {1830--1840},
 publisher = {Curran Associates, Inc.},
 title = {Outlier Robust Mean Estimation with Subgaussian Rates via Stability},
 url = {https://proceedings.neurips.cc/paper_files/paper/2020/file/13ec9935e17e00bed6ec8f06230e33a9-Paper.pdf},
 volume = {33},
 year = {2020}
}

@article{2019Diakonikolas:Kamath:Kane:Li:Moitra:Stewart,
author = {Diakonikolas, Ilias and Kamath, Gautam and Kane, Daniel and Li, Jerry and Moitra, Ankur and Stewart, Alistair},
title = {Robust Estimators in High-Dimensions Without the Computational Intractability},
journal = {SIAM Journal on Computing},
volume = {48},
number = {2},
pages = {742--864},
year = {2019},
doi = {10.1137/17M1126680}
}

@article{2018Minsker:uniform,
  author       = {Stanislav Minsker},
  title        = {Uniform Bounds for Robust Mean Estimators},
  journal      = {arXiv preprint arXiv:1812.03523},
  year         = {2018},
  eprint       = {1812.03523},
  archivePrefix= {arXiv},
  primaryClass = {math.ST},
  version      = {4},
  url          = {https://arxiv.org/abs/1812.03523v4}
}

@article{2025Minsker,
  author  = {Stanislav Minsker},
  title   = {Uniform Bounds for Robust Mean Estimators},
  journal = {Stochastic Processes and their Applications},
  volume  = {190},
  pages   = {104724},
  year    = {2025},
  doi     = {10.1016/j.spa.2025.104724}
}

@article{2002Bousquet,
  author  = {Bousquet, Olivier},
  title   = {A {Bennett} Concentration Inequality and Its Application to Suprema of Empirical Processes},
  journal = {Comptes Rendus Math{\'e}matique},
  volume  = {334},
  number  = {6},
  pages   = {495--500},
  year    = {2002},
  doi     = {10.1016/S1631-073X(02)02292-6}
}

\clearpage
\section{Supplementary Material}\label{s:supplement}

This supplement presents the omitted proofs of results stated in the main text: Lemma \ref{lemma:normal:residual:counting:v2}, Lemma \ref{lemma:normal:residual}, Corollary \ref{cor:count:linear:v3}, Corollary \ref{cor:truncated:process}, Lemma \ref{lemma:truncation:bias}, Lemma \ref{lemma:variance:auxiliary}, Proposition \ref{prop:minimax:dir:trimmed:mean}, Lemma \ref{lemma:trim:trunc}, Proposition \ref{prop:optimal:hypercontractive}. It also presents a proof diagram in Section~\ref{ss:proof:diagram}.

\subsection{Proof of Lemma \ref{lemma:normal:residual:counting:v2}}\label{ss:suplement:lemma:normal:residual:counting:v2}

Let $\sigma:=\gamma^2\Vert\bz\Vert_2^2$. Then
\begin{align}
  \Gamma_{\mathbf{0},\gamma}\unit_{\{|\langle\bz,\btheta\rangle| > B\}}
  = 2\int_{B}^\infty\frac{\exp(
      -\frac{t^2}{2\sigma}
    )}{\sqrt{2\pi \sigma}}d t
  = \frac{2\exp\left(
      -\frac{B^2}{2\sigma}
    \right)}{\sqrt{2\pi \sigma}}
    I(B,\sigma), 
\end{align}
where we have defined 
$
I(B,\sigma)
:=
\exp\left(\frac{B^2}{2\sigma}\right)
\int_B^\infty
\exp\left(-\frac{t^2}{2\sigma}\right)\,dt.
$
Hence, for the inequality stated in the lemma to hold it is enough to show
\[
  \frac{\sigma B}{B^2+\sigma}
  \le I(B,\sigma) \le \frac{\sigma}{B}.
  \label{lemma:normal:residual:counting:v2:eq1}
\]

For the upper bound, since $t\ge B$ for all $t\in[B,\infty)$,
$
1\le \frac{t}{B}.
$
Therefore,
\begin{align}
\int_B^\infty
\exp\left(-\frac{t^2}{2\sigma}\right)\,dt
\le
\frac{1}{B}
\int_B^\infty
t\exp\left(-\frac{t^2}{2\sigma}\right)\,dt
=
\frac{\sigma}{B}
\exp\left(-\frac{B^2}{2\sigma}\right).
\end{align}
Multiplying by $\exp(B^2/(2\sigma))$ yields
$
I(B,\sigma)\le \frac{\sigma}{B}.
$

For the lower bound, integration by parts gives
\begin{align}
\int_B^\infty
\exp\left(-\frac{t^2}{2\sigma}\right)\,dt
&=
\frac{\sigma}{B}
\exp\left(-\frac{B^2}{2\sigma}\right)
-
\sigma
\int_B^\infty
\frac{1}{t^2}
\exp\left(-\frac{t^2}{2\sigma}\right)\,dt.
\end{align}
Since $t\ge B$,
$
\frac{1}{t^2}\le \frac{1}{B^2},
$
and hence
\begin{align}
\int_B^\infty
\exp\left(-\frac{t^2}{2\sigma}\right)\,dt
&\ge
\frac{\sigma}{B}
\exp\left(-\frac{B^2}{2\sigma}\right)
-
\frac{\sigma}{B^2}
\int_B^\infty
\exp\left(-\frac{t^2}{2\sigma}\right)\,dt.
\end{align}
Rearranging,
\[
\left(1+\frac{\sigma}{B^2}\right)
\int_B^\infty
\exp\left(-\frac{t^2}{2\sigma}\right)\,dt
\ge
\frac{\sigma}{B}
\exp\left(-\frac{B^2}{2\sigma}\right).
\]
Therefore,
\[
\int_B^\infty
\exp\left(-\frac{t^2}{2\sigma}\right)\,dt
\ge
\frac{\sigma B}{B^2+\sigma}
\exp\left(-\frac{B^2}{2\sigma}\right).
\]
Multiplying by $\exp(B^2/(2\sigma))$ gives
$
I(B,\sigma)
\ge
\frac{\sigma B}{B^2+\sigma}.
$

This finishes the proof of \eqref{lemma:normal:residual:counting:v2:eq1}. 

\subsection{Gaussian estimates \& proof of Lemma \ref{lemma:normal:residual}}\label{appendix:ss:gaussian:lemmas}

This section develops the Gaussian estimates underlying the smoothing arguments, culminating in the residual bounds. In this section, $Q>0$ is as in Section \ref{s:truncation}. Only within this section $B:=Q^{1/q}$.

For the next two lemmas, Lemmas \ref{lemma:normal:residual:aux} and \ref{lemma:normal:residual:aux:2}, we fix $\{m_\ell\}_{\ell=1}^q\subset\re$ and $\nu>0$. Given any $\ell\in\mathbb{N}$ and  $R>0$, we define the function $f_{\ell,R}:\re^\ell\rightarrow\re$, 
\begin{align}
  f_{\ell,R}(t_1,\ldots,t_\ell) := \frac{(\prod_{j=1}^\ell|t_j| - Q)_+}{(2\pi\nu)^{\ell/2}}\exp\left\{
      -\frac{\sum_{j=1}^\ell(t_j-m_j)^2}{2\nu}
      \right\}.
\end{align}
Our first goal is to give an upper bound on the integral $\int_{\re^q}f_{q,Q}$. We start with Lemma \ref{lemma:normal:residual:aux}, giving a bound on the integral on a restricted region. 

\begin{remark}
  We will sometimes use the shorthand notations $t_{1:\ell} = (t_1,\ldots,t_\ell)$ and, in multivariate integration, $dt_{1:\ell}=dt_1\cdots d t_\ell$.
\end{remark}

\begin{lemma}\label{lemma:normal:residual:aux}
    Fix any $a,b\in(0,1)$ such that 
    $$
      \max_{\ell\in[q]}m_\ell^2\le a^2B^2 \quad\mbox{ and }\quad 
      \nu\le b^2B^2.
    $$
    Define $\theta := b^2(\nicefrac{\sqrt{\nu}}{B})\sqrt{(\nicefrac{2}{\pi})}/(1-a)^2$ and $\calI :=(-\infty,-B]\cup[B,\infty)$. Then
    $$
      \int_{\calI^q}f_{q,Q}(t_{1:q})dt_{1:q}
      \le B^q\sum_{\ell=1}^q\binom{q}{\ell}
      (\theta e^{-\frac{(1-a)^2B^2}{2\nu}})^\ell.
    $$
\end{lemma}
\begin{proof}
  By the change of variables $t_\ell\mapsto|t_\ell|$ and symmetry,
    \begin{align}
        \int_{\calI^q}f_{q,Q}(t_{1:q})dt_{1:q}
        &= \int_{[B,\infty)^q}
        \frac{\prod_{\ell=1}^qt_\ell - Q}{(2\pi\nu)^{q/2}}
        \prod_{\ell=1}^q\left(
          e^{-\frac{(t_\ell + m_\ell)^2}{2\nu}} + e^{-\frac{(t_\ell - m_\ell)^2}{2\nu}}
        \right)dt_{1:q}.
    \end{align}

  We perform the change of variables $t_\ell\rightarrow u_\ell+B$ over the regions $t_\ell\ge B$ and $u_\ell\ge0$ for all $\ell\in[q]$. Note that by the multivariate version of the binomial theorem, 
  \begin{align}
    \prod_{\ell=1}^qt_\ell - Q
    = \prod_{\ell=1}^q(u_\ell+B) - B^q
    = \sum_{\ell=1}^q B^{\,q-\ell}
    \!\!\sum_{1\le i_1<\cdots<i_\ell\le q}\!
    u_{i_1}\cdots u_{i_\ell}.
  \end{align}
  
  By Fubini's theorem, 
  \begin{align}
    &\int_{\re^q_+}
        \frac{
          u_{i_1}\cdots u_{i_r}
        }{(2\pi\nu)^{q/2}}
        \prod_{\ell=1}^q\left(
          e^{-\frac{(u_\ell + B + m_\ell)^2}{2\nu}} + e^{-\frac{(u_\ell + B - m_\ell)^2}{2\nu}}
        \right)du_{1:q}\\
    &= I\int_{\re^r_+}
        \frac{
          u_{i_1}\cdots u_{i_r}
        }{(2\pi\nu)^{r/2}}
        \prod_{j=1}^{r}\left(
          e^{-\frac{(u_{i_j} + B + m_{i_j})^2}{2\nu}} + e^{-\frac{(u_{i_j} + B - m_{i_j})^2}{2\nu}}
        \right)du_{i_1}\cdots du_{i_r}, 
  \end{align}
  with
  \begin{align}
    I &:= \prod_{\ell\in[q]\setminus\{i_1,\ldots,i_r\}}\int_{[0,\infty)}
        \frac{1}{(2\pi\nu)^{1/2}}
        \left(
          e^{-\frac{(u_\ell + B + m_\ell)^2}{2\nu}} + e^{-\frac{(u_\ell + B - m_\ell)^2}{2\nu}}
        \right)du_{\ell}\\
      &\le \prod_{\ell\in[q]\setminus\{i_1,\ldots,i_r\}}\frac{1}{2}\int_{\re}
        \frac{1}{(2\pi\nu)^{1/2}}
        \left(
          e^{-\frac{(u_\ell + B + m_\ell)^2}{2\nu}} + e^{-\frac{(u_\ell + B - m_\ell)^2}{2\nu}}
        \right)du_{\ell}
      \le 1. 
  \end{align}

  Gathering the previous facts,  
  \begin{align}
        &\int_{\calI^q}f_{q,Q}(t_{1:q})dt_{1:q}\\
        &\le \sum_{r=1}^q B^{\,q-r}
        \!\!\sum_{1\le i_1<\cdots<i_r\le q}\!
        \int_{\re^r_+}
        \frac{
          u_{i_1}\cdots u_{i_r}
        }{(2\pi\nu)^{r/2}}
        \prod_{j=1}^{r}\left(
          e^{-\frac{(u_{i_j} + B + m_{i_j})^2}{2\nu}} + e^{-\frac{(u_{i_j} + B - m_{i_j})^2}{2\nu}}
        \right)du_{i_1}\cdots du_{i_r}. 
  \end{align}

  Since $a B\ge |m_\ell|$ for any $\ell\in[q]$, we have that
  \begin{align}
    (t_\ell \pm m_\ell)^2 = u_\ell^2 + (B \pm m_\ell)^2
    + 2u_\ell(B \pm m_\ell)
    \ge (1-a)^2B^2 + 2(1-a)u_\ell B.
  \end{align}
  Consequently, $\int_{\calI^q}f_{q,Q}(t_{1:q})dt_{1:q}$ is not greater than
  \begin{align}
    &\sum_{r=1}^qB^{\,q-r}
    \!\!\sum_{1\le i_1<\cdots<i_r\le q}\!
    \frac{
      2^{r} e^{-r\frac{(1-a)^2B^2}{2\nu}}
    }{(2\pi\nu)^{r/2}}  
    \int_{\re_+^{r}}
    \left(\prod_{j=1}^{r} u_{i_j}e^{-\frac{(1-a)u_{i_j} B}{\nu}}\right)
    du_{i_1:i_r}\\
    &= \sum_{r=1}^qB^{\,q-r}
    \!\!\sum_{1\le i_1<\cdots<i_r\le q}\!
    \left(
      \frac{
        2e^{-\frac{(1-a)^2B^2}{2\nu}}
      }{\sqrt{2\pi\nu}}\int_{0}^\infty ue^{-\frac{(1-a)u B}{\nu}}du
    \right)^{r}. 
  \end{align}

  For any $\eta>0$, 
  $
    \int_{0}^\infty ue^{-\frac{(1-a)u B}{\nu}}du
    = \frac{\nu^2}{(1-a)^2B^2}.
  $
  Define the constant
  \begin{align}
    L := \frac{
        2e^{-\frac{(1-a)^2B^2}{2\nu}}
      }{\sqrt{2\pi\nu}}\cdot\frac{\nu^2}{(1-a)^2B^2}
      \le \sqrt{\frac{2}{\pi}}\exp\left(
      -\frac{(1-a)^2B^2}{2\nu}
    \right)\frac{b^2(\nicefrac{\sqrt{\nu}}{B})}{(1-a)^2}B,
  \end{align}
  where we used that 
  $
  b^2 B^2\ge \nu
      \Rightarrow
      \frac{\nu^{\frac{3}{2}}}{B^2}
      \le b^2B\cdot\frac{\sqrt{\nu}}{B}.
  $
  
  By the multivariate binomial theorem, 
  \begin{align}
    \int_{\calI^q}f_{q,Q}(t_{1:q})dt_{1:q}
    \le 
    \sum_{r=1}^qB^{\,q-r}
    \!\!\sum_{1\le i_1<\cdots<i_r\le q}\!
    L^{r}
    = (L+B)^q - B^q
    = \sum_{r=1}^q\binom{q}{r}L^r B^{q-r}.
  \end{align}
  Using the bound on $L$ above completes the proof. 
\end{proof}

Using Lemma \ref{lemma:normal:residual:aux}, we now prove Lemma \ref{lemma:normal:residual:aux:2}, stating a bound on $\int_{\re^q}f_{q,Q}$.
\begin{lemma}\label{lemma:normal:residual:aux:2}
  Suppose assumptions of Lemma \ref{lemma:normal:residual:aux} hold and $b^3\le(1-a)^2$. Then 
  $$
    \int_{\re^q}f_{q,Q} \le 4^q Q \theta e^{-\frac{(1-a)^2B^2}{2\nu}}.
  $$
\end{lemma}
\begin{proof}
  For every $\ell\in[q]$ and $\{i_1\ldots,i_\ell\}\subset[q]$, we define the set
\begin{align}
  \calR_{\{i_1\ldots,i_\ell\}} := \left\{
    t_{1:q}\in\re^q : 
    |t_i|>B, \forall i\in\{i_1,\ldots,i_\ell\},
    |t_i|\le B, \forall i\notin\{i_1,\ldots,i_\ell\}
  \right\}.
\end{align}
This family of subsets of $\re^q$ forms a partition. Hence
$$
  \int_{\re^q}f_{q,Q} = 
  \sum_{\ell\in[q],\{i_1\ldots,i_\ell\}\subset[q]}
  \int_{\calR_{\{i_1\ldots,i_\ell\}}}f_{q,Q}.
$$
In the region $\calR_{\{i_1\ldots,i_\ell\}}$, 
$
\prod_{\ell=1}^{q}|t_\ell| - Q
\le B^{q-\ell}(
  \prod_{j=1}^{\ell}|t_{i_j}| - B^{\ell}
).
$
It follows that 
$
  \int_{\calR_{\{i_1\ldots,i_\ell\}}} f_{q,Q}
  \le B^{q-\ell}\int_{\calI^\ell} f_{\ell,B^\ell}.
$
By Lemma \ref{lemma:normal:residual:aux}, 
\begin{align}
  \int_{\re^q}f_{q,Q} &\le 
  \sum_{\ell\in[q],\{i_1\ldots,i_\ell\}\subset[q]}
  B^{q-\ell}\int_{\calI^\ell} f_{\ell,B^\ell}\\
  &\le \sum_{\ell\in[q],\{i_1\ldots,i_\ell\}\subset[q]}
      B^{q-\ell}B^\ell\sum_{i=1}^\ell\binom{\ell}{i}
      \left(\theta e^{-\frac{(1-a)^2B^2}{2\nu}}\right)^i\\
  &\le \theta e^{-\frac{(1-a)^2B^2}{2\nu}}
  B^q\sum_{i=1}^q\binom{q}{i}\sum_{\ell=0}^q\binom{q}{\ell}\\
  &\le 4^qB^q \theta e^{-\frac{(1-a)^2B^2}{2\nu}}.
\end{align}
In the last inequality, we first used that $\theta e^{-\frac{(1-a)^2B^2}{2\nu}}\le \theta \le1$ since $b\le\sqrt{\nu}/B$ and $b^3\le(1-a)^2$; then we used that $\sum_{\ell=0}^q\binom{q}{\ell}=2^q$. 
\end{proof}

Using Lemma \ref{lemma:normal:residual:aux:2}, we obtain the proof of Lemma \ref{lemma:normal:residual} in the main text. We state a version of this lemma with general constants. Lemma \ref{lemma:normal:residual} follows from it with $a=b=\frac{1}{2}$.

\begin{lemma}\label{lemma:normal:residual:general}
Let
    $
    \bfU:=\{\bu_{j,\ell}\}_{j\in[m],\ell\in[q]}
    \in(\re^d)^{m\times q},
    $
    $\bz\in\re^d$, $\gamma>0$ and $A\ge1$. Let
    $
    \nu:=\gamma^2\Vert\bz\Vert_2^2
    $
    and
    $
    m_{j,\ell}:=\langle\bz,\bu_{j,\ell}\rangle
    $
    for $j\in[m]$ and $\ell\in[q]$. Fix any $a,b\in(0,1)$ such that $b^3\le(1-a)^2$ and 
    \begin{align}
      \max_{j\in[m],\ell\in[q]}
      m_{j,\ell}^2
      \le a^2 \left( \frac{Q}{m} \right)^{2/q}
      \quad\mbox{ and }\quad
      \nu \le b^2 \left(\frac{Q}{m}\right)^{2/q}. 
    \end{align}
    Define
    $
    \theta' := \frac{b^2\sqrt{(\nicefrac{2}{\pi})}}{(1-a)^2}
    \cdot
    \frac{\sqrt{\nu}}
    {\left(
      \frac{Q}{m}
    \right)^{\nicefrac{1}{q}}}.
    $

Then
\begin{align}
\Bigg|
\sum_{j=1}^{m}
\prod_{\ell=1}^{q}
\langle\bz,\bu_{j,\ell}\rangle
-
\Gamma_{\bfU,\gamma}
\Bigg[
\psi_{AQ}
\Bigg(
\sum_{j=1}^{m}
\prod_{\ell=1}^{q}
\langle\bz,\btheta_{j,\ell}\rangle
\Bigg)
\Bigg]
\Bigg|
\le \frac{4^qQ}{2}
\theta'\exp\left(
-\frac{(1-a)^2}{2}
\frac{
\left(\frac{Q}{m}\right)^{2/q}
}{\nu}
\right).
\end{align}
\end{lemma}
\begin{proof}
  Let $\nu:=\gamma^2\Vert\bz\Vert_2^2$ and, for each
  $j\in[m]$, $\ell\in[q]$, let
  $m_{j,\ell}:=\langle\bz,\bu_{j,\ell}\rangle$
  and
  $N_{j,\ell}\sim\calN(m_{j,\ell},\nu)$,
  where 
  $\{N_{j,\ell}\}_{j\in[m],\ell\in[q]}$
  is a family of independent normal random variables. Hence
  \begin{align}
    \Gamma_{\bfU,\gamma}\left[
      \psi_{AQ}\left(
      \sum_{j=1}^{m}
      \prod_{\ell=1}^{q}
      \langle\bz,\btheta_{j,\ell}\rangle
      \right)
    \right]
    = \esp\left[
      \psi_{AQ}\left(
      \sum_{j=1}^{m}
      \prod_{\ell=1}^{q}
      N_{j,\ell}
      \right)
    \right],
  \end{align}
  and
  $
    \esp\left[
      \sum_{j=1}^{m}
      \prod_{\ell=1}^{q}
      N_{j,\ell}
    \right]
    =
    \sum_{j=1}^{m}
    \prod_{\ell=1}^{q}
    \langle\bz,\bu_{j,\ell}\rangle.
  $

  For any $t\in\re$,
  $
    t-\psi_{AQ}(t)
    =
    (t-AQ)_+
    -
    (t+AQ)_-,
  $
  and, in particular,
  $
    |t-\psi_{AQ}(t)|
    \le
    (|t|-AQ)_+
    \le
    (|t|-Q)_+,
  $
  where we used that $A\ge1$.
  Therefore,
  \begin{align}
    &
    \left|
    \esp\left[
      \sum_{j=1}^{m}
      \prod_{\ell=1}^{q}
      N_{j,\ell}
      -
      \psi_{AQ}\left(
      \sum_{j=1}^{m}
      \prod_{\ell=1}^{q}
      N_{j,\ell}
      \right)
    \right]
    \right|
    \\
    &\le
    \int_{\re^{mq}}
    \frac{
      \left(
      \sum_{j=1}^{m}
      \prod_{\ell=1}^{q}
      |t_{j,\ell}|
      -
      Q
      \right)_+
    }{
      (2\pi\nu)^{\frac{mq}{2}}
    }
    \exp\left\{
      -\frac{
      \sum_{j=1}^{m}
      \sum_{\ell=1}^{q}
      (t_{j,\ell}-m_{j,\ell})^2
      }{
      2\nu
      }
    \right\}
    dt_{1:m,1:q} =:J. \label{lemma:normal:residual:general:eq1}
  \end{align}

  Using that $(t-Q)_+\le(t-Q/m)_+$ and 
  $(\sum_{j=1}^{m}A_j)_+
  \le
  \sum_{j=1}^{m}(A_j)_+$
  for any sequence
  $\{A_j\}_{j=1}^{m}\subset\re$,
  we obtain:
  \begin{align}
    J  &\le
    \sum_{j=1}^{m}
    \int_{\re^{mq}}
    \frac{
      \left(
      \prod_{\ell=1}^{q}
      |t_{j,\ell}|
      -
      \frac{Q}{m}
      \right)_+
    }{
      (2\pi\nu)^{\frac{mq}{2}}
    }
    \exp\left\{
      -\frac{
      \sum_{r=1}^{m}
      \sum_{\ell=1}^{q}
      (t_{r,\ell}-m_{r,\ell})^2
      }{
      2\nu
      }
    \right\}
    dt_{1:m,1:q}
    \\
    &\le
    \sum_{j=1}^{m}
    \int_{\re^{q}}
    \frac{
      \left(
      \prod_{\ell=1}^{q}
      |t_{j,\ell}|
      -
      \frac{Q}{m}
      \right)_+
    }{
      (2\pi\nu)^{\frac{q}{2}}
    }
    \exp\left\{
      -\frac{
      \sum_{\ell=1}^{q}
      (t_{j,\ell}-m_{j,\ell})^2
      }{
      2\nu
      }
    \right\}
    dt_{j,1:q}.
    \label{lemma:normal:residual:general:eq2}
  \end{align}

  Recall the definition of $\theta'$. For every
  $r\in[m]$,
  $
    \max_{\ell\in[q]} m_{r,\ell}^2 \le a^2\left(\nicefrac{Q}{m}\right)^{\frac{2}{q}},
  $
  $
    \nu \le b^2\left(\nicefrac{Q}{m}\right)^{\frac{2}{q}}.
  $
  We apply Lemma \ref{lemma:normal:residual:aux:2}
  to each integral in
  \eqref{lemma:normal:residual:general:eq2}
  with these parameters.
  Substituting these bounds in
  \eqref{lemma:normal:residual:general:eq2}
  and
  \eqref{lemma:normal:residual:general:eq1},
  we prove the claim.
\end{proof}

\subsection{Proof of Corollary \ref{cor:count:linear:v3}}\label{ss:suplement:cor:count:linear:v3}

  We invoke Lemma \ref{lemma:count:linear:v3}
  with $c_1\leftarrow c_1/(mq)$,
  $A=4^q$
  and
  $\gamma^{-2}:=\erank(\bfSigma)$. 
    
  Condition
  \eqref{cor:count:linear:condition:B:c:v3}
  implies condition
  \eqref{lemma:count:linear:condition:B:c:v3}
  of Lemma \ref{lemma:count:linear:v3}
  with such parameters.

  Finally, $\delta\in(0,5/7)$ implies that $56 + 8\log(1/\delta)\le 175\log(1/\delta)$. Using also that $k\ge (c_2^{-1}\log(1/\delta)) \vee (c_3^{-1}\erank(\bfSigma))$, the claim follows from Lemma \ref{lemma:count:linear:v3}.

\subsection{Proof of Corollary \ref{cor:truncated:process}}
\label{ss:supplement:cor:truncated:process}

  We will apply Lemma \ref{lemma:truncated:process} with the two distinct cases:
  \begin{align}
    \begin{cases}
      p \ge 2q, &\mbox{ taking } r:=2q,\\
      q\le p < 2q, &\mbox{ taking } r:=p.
    \end{cases}
  \end{align}
  In both cases we apply Lemma \ref{lemma:truncated:process} with $c_1\leftarrow c_1/(mq)$, $A=4^q$ and $\gamma^{-2}:=\erank(\bfSigma)$.

  Condition
  \eqref{cor:count:linear:condition:B:c:v3}
  implies
  \eqref{lemma:count:linear:condition:B:c:v3}
  of Lemma \ref{lemma:truncated:process}
  with such parameters. Since
  $\delta\in(0,5/7)$, 
  $36+4\log 2+8\log(2/\delta)\le 140\log(1/\delta)$, 
  $2+2\log 2\le 11\log(1/\delta)$
  and $2\log(2/\delta)\le 7\log(1/\delta)$. From this fact, $k\ge (c_2^{-1}\log(1/\delta)) \vee (c_3^{-1}\erank(\bfSigma))$, 
  $\nu_r\le\|\bfSigma\|^{1/2}\kappa_r$, $\nu_r(\gamma)\le2\|\bfSigma\|^{1/2}\kappa_r$, $q\le r\le2q$ and Lemma \ref{lemma:truncated:process} we obtain that, with probability at least $1-\delta$,
  \begin{align}
  \epsilon_{(4^q Q)} &\le 3m^{\frac{r}{2q}}(4^qQ)^{1-\frac{r}{2q}}
    (4\|\bfSigma\|^{1/2}\kappa_r)^{\frac r2}
    \sqrt{\frac{\erank(\bfSigma)+11\log(1/\delta)}{n}}\\ 
    &+ \left[
    2\left(
        20 + \frac{10+\sqrt{\frac{\pi}{2}}}{mq}
    \right)c_1 + \frac{140c_2 + 18c_3}{3}
  \right](4^qQ)\frac{k}{n}. \label{cor:truncated:process:eq1}
\end{align}

We now split our argument between the two cases. If $p\ge2q$, then $r=2q$ and
\begin{align}
  m^{\frac{r}{2q}}(4^qQ)^{1-\frac{r}{2q}} (4\|\bfSigma\|^{1/2}\kappa_r)^{\frac{r}{2}}
  = m(4\|\bfSigma\|^{1/2}\kappa_{2q})^q. \label{cor:truncated:process:eq2}
\end{align}
If $q\le p < 2q$, then $r=p$ and by Young's inequality,  
\begin{align}
  &m^{\frac{r}{2q}}(4^qQ)^{1-\frac{r}{2q}} 
  (4\|\bfSigma\|^{1/2}\kappa_r)^{\frac{r}{2}}
  \sqrt{\frac{\erank(\bfSigma) + 11\log(1/\delta)}{n}}\\
  &= \left(
    4^qQ \frac{\erank(\bfSigma) + 11\log(1/\delta)}{n}
  \right)^{1-\frac{r}{2q}}
  \left(
    m(4\|\bfSigma\|^{1/2}\kappa_r)^{q} \left(
      \frac{\erank(\bfSigma) + 11\log(1/\delta)}{n}
    \right)^{1-\frac{q}{r}}
  \right)^{\frac{r}{2q}}\\
  &\le \left(
    1 - \frac{r}{2q}
  \right) 4^qQ \frac{\erank(\bfSigma)+11\log(1/\delta)}{n}
  + \frac{r}{2q}\cdot m(4\|\bfSigma\|^{1/2}\kappa_r)^{q} \left(
      \frac{\erank(\bfSigma)+11\log(1/\delta)}{n}
    \right)^{1-\frac{q}{r}}\\
  &\le 4^qQ \frac{(11c_2+c_3)k}{n}
  + m(4\|\bfSigma\|^{1/2}\kappa_r)^{q} \left(
      \frac{(11c_2+c_3)k}{n}
    \right)^{1-\frac{q}{p}}. 
    \label{cor:truncated:process:eq3}
\end{align}

Using the bounds \eqref{cor:truncated:process:eq2}-\eqref{cor:truncated:process:eq3} in \eqref{cor:truncated:process:eq1} for each of the two cases, finishes the proof. 

\subsection{Proof of Lemma \ref{lemma:truncation:bias}}
\label{ss:supplement:lemma:truncation:bias}

For any $Z\in\re$ and $t>0$, 
$
  |Z-\psi_{t}(Z)|\le (|Z|-t)_+.
$
In addition, 
$(\sum_{j=1}^{m}Z_j)_+ \le \sum_{j=1}^{m}(Z_j)_+$ for any sequence 
$\{Z_j\}_{j=1}^{m}\subset\re$. Hence, 
\begin{align}
  \calT_t &= \sup_{\bfU\in\mbB_2^{m\times q}}\esp\left[
      \left(
      \sum_{j=1}^{m}
      \prod_{\ell=1}^{q}
      \langle\bx,\bu_{j,\ell}\rangle
      \right)
      -\psi_{t}\!\left(
      \sum_{j=1}^{m}
      \prod_{\ell=1}^{q}
      \langle\bx,\bu_{j,\ell}\rangle
      \right)\right]\\
  &\le \sup_{\bfU\in\mbB_2^{m\times q}}\esp\left[
      \left(
      \left|\sum_{j=1}^{m}
      \prod_{\ell=1}^{q}
      \langle\bx,\bu_{j,\ell}\rangle\right|
      - t
      \right)_+\right]\\
  &\le \sum_{j=1}^{m}\sup_{\bfU\in\mbB_2^{m\times q}}\esp\left[
      \left(
      \prod_{\ell=1}^{q}
      \left|\langle\bx,\bu_{j,\ell}\rangle\right|
      - \frac{t}{m}
      \right)_+\right]. 
\end{align}

Since $p\ge q$, for any $Z\in\re$ and $t>0$, $(|Z|-t)_+\le\frac{|Z|^{\frac{p}{q}}}{t^{\frac{p}{q}-1}}$. Hence, 
\begin{align}
  \calT_t &\le \sum_{j=1}^m\sup_{\bfU\in\mbB_2^{m\times q}}\esp\left[
    \frac{\prod_{\ell=1}^{q}
      |\langle\bx,\bu_{j,\ell}\rangle|^{\frac{p}{q}}}{(t/m)^{\frac{p}{q}-1}}
  \right]\\
  &\le \frac{1}{(t/m)^{\frac{p}{q}-1}}
  \sum_{j=1}^m\sup_{\bfU\in\mbB_2^{m\times q}}\prod_{\ell=1}^{q}
      \left(
      \esp\!\left[
      |\langle\bx,\bu_{j,\ell}\rangle|^{p}
      \right]
      \right)^{\frac{1}{q}}\\
  &= \frac{m^{\frac{p}{q}-1}}{t^{\frac{p}{q}-1}}
  \sum_{j=1}^m
  \sup_{\bfU\in\mbB_2^{m\times q}}\prod_{\ell=1}^{q}
      \left(
      \esp\!\left[
      |\langle\bx,\bu_{j,\ell}\rangle|^{p}
      \right]
      \right)^{\frac{1}{q}}
\end{align}
where in the last inequality we used the generalized Cauchy--Schwarz inequality. For every $\bu\in\mbB_2$,
    $
      \esp\!\left[
      |\langle\bx,\bu\rangle|^{p}
      \right]
      \le \nu_{p}^{p}.
    $
Gathering the previous bounds, we finish the proof of the first stated inequality. 

The second inequality stated in the lemma follows immediately from the first and the fact that 
$
    \frac{Q^{1/q}}
    { 2m^{\nicefrac{1}{q}}(mq)^{\nicefrac{1}{p}} }
    \ge \left(\frac{\nu_p^pn}{c_1k}\right)^{\nicefrac{1}{p}}
$
by condition \eqref{cor:count:linear:condition:B:c:v3}. 

\subsection{Proof of Lemma \ref{lemma:variance:auxiliary}}
\label{ss:supplement:lemma:variance:auxiliary}
For any $Z\in\re$, $\psi_{AQ}^2(Z)=(|Z|\wedge (AQ))^2$. Hence, 
\begin{align}
  \esp\left[
  \psi_{AQ}^2\left(\sum_{j=1}^m\prod_{\ell=1}^q X_{\ell,j}\right)
  \right] &= 
  \esp\left[
    \left(
      \left|\sum_{j=1}^m\prod_{\ell=1}^q X_{\ell,j}\right|\wedge (AQ)
    \right)^2
  \right]\\
  &= \esp\left[
    \left(
      \left|\sum_{j=1}^m\prod_{\ell=1}^q X_{\ell,j}\right|\wedge (AQ)
    \right)^{2-\frac{r}{q}}
    \left(
      \left|\sum_{j=1}^m\prod_{\ell=1}^q X_{\ell,j}\right|\wedge (AQ)
    \right)^{\frac{r}{q}}
  \right]\\
  &\le (AQ)^{2-\frac{r}{q}}
  \esp\left[
    \left(
      \left|\sum_{j=1}^m\prod_{\ell=1}^q X_{\ell,j}\right|
    \right)^{\frac{r}{q}} 
  \right]\\
  &\le m^{r/q-1} (AQ)^{2-\frac{r}{q}}
  \sum_{j=1}^m\esp\left[
      \left|\prod_{\ell=1}^q X_{\ell,j}\right|^{\frac{r}{q}}
  \right]\\
  &\le m^{r/q-1} (AQ)^{2-\frac{r}{q}}
  \sum_{j=1}^m\prod_{\ell=1}^q
      \left(\esp\left[
      \left|X_{\ell,j}\right|^{r}
  \right]\right)^{\frac{1}{q}}.
\end{align}
In the first inequality we used that $0\le r\le2q$. In the second inequality we used that the function $t\mapsto|t|^{\frac{r}{q}}$ is convex since $r\ge q$ and, in the third we used the generalized H\"older inequality.

\subsection{Proof of Proposition \ref{prop:minimax:dir:trimmed:mean}}
\label{ss:supplement:prop:minimax:dir:trimmed:mean}

  We omit the details regarding existence and measurability of $\widehat\bfT_k$; see e.g. the details in the proof of Proposition 3.1 in \cite{2024Oliveira:Rico}. For the inequality, an analogous argument in the mentioned proof entails 
  \begin{align}
    \|\widehat\bfT_k-\esp[\bx^{\otimes q}]\| &=
    \sup_{\bu\in\mbS_2}\left|
      \langle\widehat\bfT_k,\bu^{\otimes q}\rangle 
      - \esp[\langle\bx,\bu\rangle^q]
    \right|\\
    &\le \sup_{\bu\in\mbS_2}\left|
      \langle\widehat\bfT_k,\bu^{\otimes q}\rangle - \sfT_k(\tilde\bx_{1:n}|\bu)
    \right|
    + \sup_{\bu\in\mbS_2}\left|
      \esp[\langle\bx,\bu\rangle^q] - \sfT_k(\tilde\bx_{1:n}|\bu)
    \right|\\
    &\le 2\sup_{\bu\in\mbS_2}\left|
      \esp[\langle\bx,\bu\rangle^q] - \sfT_k(\tilde\bx_{1:n}|\bu)
    \right|.
  \end{align}
This finishes the proof. 

\subsection{Proof of Lemma \ref{lemma:trim:trunc}}
\label{ss:supplement:lemma:trim:trunc}

Let $k\in [\lfloor(n-1)/2\rfloor]$. In the following we assume the event $\Count_q(\bz_{1:n},Q,k)$ holds. On this event, in particular we have, for all $\bu\in\mbS_2$, 
\begin{align}
  \#\{i\in[n]:\langle\bz_i,\bu\rangle^q > Q\} \le k
  ~~~\mbox{ and }~~~
  \#\{i\in[n]:\langle\bz_i,\bu\rangle^q < -Q\} \le k. 
\end{align}
We refer to these as the \emph{counting conditions}. 

Fix $\bu\in\mbS_2$ and denote by $I_{\bu}\subset[n]$ the index set retained by the trimmed mean $\sfT_k(\bz_{1:n}\mid\bu)$. Thus, \(|I_{\bu}|=n-2k\) and 
$$
  \sfT_k(\bz_{1:n}\mid\bu) := \frac{1}{n-2k}
  \sum_{i=k+1}^{n-k}\langle\bz,\bu\rangle_{(i)}^{\,q}
  = \frac{1}{n-2k}
  \sum_{i\in I_{\bu}}\langle\bz_i,\bu\rangle^{\,q}.
$$

By the counting conditions and the fact that the \(k\) largest and \(k\) smallest observations are removed, it follows that 
$
  |\langle\bz_i,\bu\rangle^{\,q}| \le Q
$
for any $i\in I_{\bu}$. In particular, 
$
  \psi_Q(\langle\bz_i,\bu\rangle^{\,q}) = \langle\bz_i,\bu\rangle^{\,q}
$
for any $i\in I_{\bu}$. Hence, we can write:
\begin{align}
  n \sfT_Q(\bz_{1:n}\mid\bu) := \sum_{i=1}^n\psi_Q(\langle\bz_i,\bu\rangle^{q})
  = (n-2k)\sfT_k(\bz_{1:n}\mid\bu) 
  + \sum_{i\notin I_{\bu}}\psi_Q(\langle\bz_i,\bu\rangle^{q}). 
  \label{ss:supplement:lemma:trim:trunc:eq1}
\end{align}
Since $|\psi_Q(t)|\le Q$ for any $t\in\re$ and $|[n]\setminus I_{\bu}|=2k$, we can also write
\begin{align}
  \left|
    \sum_{i\notin I_{\bu}}\psi_Q(\langle\bz_i,\bu\rangle^{q})
  \right| \le 2kQ\mbox{ and }|2k\sfT_k(\bz_{1:n}\mid\bu)|\leq 2kQ.\label{ss:supplement:lemma:trim:trunc:eq2}
\end{align}

From \eqref{ss:supplement:lemma:trim:trunc:eq1} and \eqref{ss:supplement:lemma:trim:trunc:eq2} it follows that
\begin{align}
  \Big|
    \frac{n}{n-2k} \sfT_Q(\bz_{1:n}\mid\bu) - \sfT_k(\bz_{1:n}\mid\bu)
  \Big| \le \frac{2kQ}{n-2k}. 
\end{align}
Moreover, 
\begin{align}
\left|\sfT_k(\bz_{1:n}\mid\bu) -\sfT_Q(\bz_{1:n}\mid\bu)\right| 
& \leq \frac{1}{n}\left|(n-2k)\sfT_k(\bz_{1:n}\mid\bu) -n\sfT_Q(\bz_{1:n}\mid\bu)\right|  + \frac{|2k\sfT_k(\bz_{1:n}\mid\bu)|}{n}\\ 
& = \frac{2k}{n}\left|\sum_{i\notin I_{\bu}}\psi_Q(\langle\bz_i,\bu\rangle^{q})\right| + \frac{|2k\sfT_k(\bz_{1:n}\mid\bu)|}{n} 
\leq \frac{4kQ}{n}
\end{align}
This completes the proof. 

\subsection{Proof of Proposition \ref{prop:optimal:hypercontractive}}
\label{ss:supplement:prop:optimal:hypercontractive}

Set 
$
\alpha:=\frac{\epsilon}{1-\epsilon}
$ 
and define the distributions 
$
  P_0:=\delta_1
$ 
and 
$
P_1:=(1-\alpha)\delta_1+\alpha\delta_a,
$
where 
$
a := \left( \nicefrac{(\kappa^p-1+\alpha)}{\alpha} \right)^{1/p}.
$ 
Clearly, $P_0\in\calF_{p,\kappa}$. Moreover, 
$
\esp_{P_1}|X|^p = 1-\alpha+\alpha a^p = \kappa^p, 
$ 
whereas 
$
\esp_{P_1}|X|^2 = 1-\alpha+\alpha a^2 \ge1.
$ 
Consequently, 
$
\{\esp_{P_1}|X|^p\}^{1/p} = \kappa \le \kappa\{\esp_{P_1}|X|^2\}^{1/2}
$ 
and therefore $P_1\in\calF_{p,\kappa}$.

The corresponding parameters satisfy
\begin{align}
\esp_{P_1}|X|^q-\esp_{P_0}|X|^q
=
\alpha(a^q-1)
=
\alpha^{1-\frac{q}{p}}
(\kappa^p-1+\alpha)^{\frac{q}{p}}
-\alpha.
\end{align}
Since $\kappa^p-1+\alpha\ge\kappa^p-1$, we obtain
\begin{align}
\esp_{P_1}|X|^q-\esp_{P_0}|X|^q
\ge
\alpha^{1-\frac{q}{p}}
(\kappa^p-1)^{\frac{q}{p}}
-\alpha\
=
\alpha^{1-\frac{q}{p}}
(\kappa^p-1)^{\frac{q}{p}}
\left[
1-
\left(
\frac{\alpha}{\kappa^p-1}
\right)^{\frac{q}{p}}
\right].
\end{align}
Therefore, if
$
\alpha
\le
2^{-\frac{p}{q}}(\kappa^p-1),
$
then
$
\left(
\frac{\alpha}{\kappa^p-1}
\right)^{\frac{q}{p}}
\le \frac12,
$
and consequently, since $\alpha\ge\epsilon$, 
\begin{align}
\esp_{P_1}|X|^q-\esp_{P_0}|X|^q
\ge
\frac12
(\kappa^p-1)^{\frac{q}{p}}
\alpha^{1-\frac{q}{p}}
\ge \frac12
(\kappa^p-1)^{\frac{q}{p}}
\epsilon^{1-\frac{q}{p}}.
\end{align}

Now let
$
  Q_0:=\delta_a
$ 
and 
$
Q_1:=\delta_1. 
$
For the first contaminated model,
$
(1-\epsilon)P_0+\epsilon Q_0
=
(1-\epsilon)\delta_1+\epsilon\delta_a.
$
For the second one,
\begin{align}
(1-\epsilon)P_1+\epsilon Q_1
&=
(1-\epsilon)
\left[
(1-\alpha)\delta_1+\alpha\delta_a
\right]
+\epsilon\delta_1\\
&=
\left[
(1-\epsilon)(1-\alpha)+\epsilon
\right]\delta_1
+
(1-\epsilon)\alpha\,\delta_a.
\end{align}
Moreover, since $(1-\epsilon)\alpha = \epsilon$, 
$
(1-\epsilon)(1-\alpha)+\epsilon
= 1-\epsilon.
$
Hence
$
(1-\epsilon)P_1+\epsilon Q_1
=
(1-\epsilon)\delta_1+\epsilon\delta_a
=
(1-\epsilon)P_0+\epsilon Q_0.
$
Thus, the two contaminated distributions coincide exactly, even though
their target parameters differ by at least
$
\frac12
(\kappa^p-1)^{\frac{q}{p}}
\epsilon^{1-\frac{q}{p}}.
$
The standard two-point indistinguishability argument therefore yields
the claimed minimax lower bound. See Theorem 5.1 in \cite{2018Chen:Gao:Ren}. 

\subsection{Proof diagram}\label{ss:proof:diagram}

\begin{center}

\scriptsize

\begin{tikzpicture}[
    x=1cm,
    y=1cm,
    >=Stealth,
    box/.style={
        draw,
        rounded corners=2pt,
        align=center,
        minimum width=3.05cm,
        minimum height=.72cm,
        inner sep=2.2pt,
        fill=white,
        font=\scriptsize
    },
    widebox/.style={
        draw,
        rounded corners=2pt,
        align=center,
        minimum width=3.35cm,
        minimum height=.72cm,
        inner sep=2.2pt,
        fill=white,
        font=\scriptsize
    },
    group/.style={
        draw=gray!65,
        rounded corners=4pt,
        fill=gray!7,
        inner sep=5pt
    },
    arrow/.style={
        ->,
        dashed,
        gray!90,
        line width=.45pt
    }
]


\node[box] (tail) at (-5.45,0)
{Lemma \ref{lemma:normal:residual:counting:v2}\\
1-d Gaussian tail a.s.};

\node[box] (exp) at (-5.45,-1.45)
{Lemma \ref{lemma:exp:Q}\\
Smoothing residual\\
(expectation)};

\node[box] (proj) at (-5.45,-3.15)
{Proposition \ref{prop:counting:linear:v2}\\
Smoothed \\ 
counting probability\\
(expectation)};

\node[box] (count) at (-5.45,-4.9)
{Lemma \ref{lemma:count:linear:v2}\\
Counting lemma\\
(expectation)};

\node[box] (aux1) at (1.80,0)
{Lemma \ref{lemma:normal:residual:aux}\\
Gaussian tail on $\calI^q$};

\node[box] (aux2) at (1.80,-1.45)
{Lemma \ref{lemma:normal:residual:aux:2}\\
Gaussian tail on $\re^q$};

\node[box] (main1) at (-1.80,-9.95)
{Lemma \ref{lemma:count:linear:v3}\\
Counting process\\
(concentration)};

\node[box] (cor1) at (-1.80,-11.65)
{Corollary \ref{cor:count:linear:v3}\\
$\gamma^{-2}=\erank(\bfSigma)$};

\node[box] (res2) at (1.80,-3.25)
{Lemma \ref{lemma:normal:residual}\\
Truncation residual a.s.};

\node[box] (diff2) at (1.80,-5.0)
{Lemma \ref{lemma:smoothed:trunc:process:diff}\\
Truncation residual\\
(expectation)};

\node[box] (smooth2) at (1.80,-6.55)
{Lemma \ref{lemma:smoothed:trunc:process}\\
Smoothed truncated \\
process\\
(expectation)};

\node[box] (prop2) at (1.80,-8.25)
{Proposition \ref{prop:truncated:complexity}\\
Truncated process\\
(expectation)};

\node[box] (main2) at (1.80,-9.95)
{Lemma \ref{lemma:truncated:process}\\
Truncated process\\
(concentration)};

\node[box] (cor2) at (1.80,-11.65)
{Corollary \ref{cor:truncated:process}\\
$\gamma^{-2}=\erank(\bfSigma)$};

\node[widebox] (pac) at (-5.45,-6.55)
{Proposition \ref{prop:bernstein:smoothed}\\
PAC-Bayesian \\ 
Bernstein inequality\\
(expectation)};

\node[widebox] (uniform) at (-5.45,-9.95)
{Uniform Bernstein \\ 
inequality};


\begin{scope}[on background layer]

\node[group,
      fit=(aux1)(aux2),
      inner xsep=4pt,
      inner ysep=5pt,
      label={[font=\scriptsize]above:Gaussian residual lemmas}] {};

\node[group,
      fit=(tail)(exp)(proj)(count),
      inner xsep=5pt,
      inner ysep=5pt,
      label={[font=\scriptsize]above:Counting lemma \& smoothing residual}] {};

\node[group,
      fit=(main1)(cor1),
      label={[font=\scriptsize]above:Counting process}] {};

\node[group,
      fit=(diff2)(smooth2)(prop2)(main2)(cor2),
      label={[font=\scriptsize]above:Truncated process}] {};

\end{scope}


\draw[arrow] (aux1)--(aux2);

\draw[arrow] (aux2)--(res2);

\draw[arrow] (tail)--(proj);

\draw[arrow] (exp)--(proj);

\draw[arrow] (proj)--(count);

\draw[arrow] (count)--(main1);

\draw[arrow] (main1)--(cor1);

\draw[arrow] (res2)--(diff2);

\draw[arrow] (diff2)--(prop2);

\draw[arrow] (smooth2)--(prop2);

\draw[arrow] (count)--(prop2);

\draw[arrow] (exp)--(prop2);

\draw[arrow] (prop2)--(main2);

\draw[arrow] (main2)--(cor2);

\draw[arrow] (pac)--(smooth2);

\draw[arrow] (pac)--(count);

\draw[arrow] (uniform)--(main1);

\draw[arrow] (uniform)--(main2);

\end{tikzpicture}

\end{center}

\end{document}